\documentclass[twoside]{amsart}
\usepackage{amsfonts, amsmath, amssymb}
\usepackage[linktocpage=true, colorlinks=true, linkcolor=magenta, filecolor=magenta,urlcolor=cyan,citecolor=cyan]{hyperref}
\usepackage{graphicx}
\usepackage{float}
\usepackage{srcltx}
\usepackage[all]{xy}
\usepackage{bbm}
\usepackage[T2A]{fontenc}
\usepackage[utf8]{inputenc}	
\usepackage{CJKutf8}
\usepackage{xcolor}
\usepackage{enumerate}
\usepackage{enumitem}
\usepackage[normalem]{ulem}
\usepackage{bm}
\usepackage{pifont}
\usepackage{latexsym}
\usepackage{stmaryrd}
\usepackage[2emode]{psfrag}
\usepackage{yhmath}
\usepackage{array}
\usepackage{dsfont}
\usepackage{mathrsfs}
\usepackage{tikz-cd}
\usepackage{comment}
\usepackage{subcaption}
\usepackage{aligned-overset}
\usetikzlibrary { decorations.pathmorphing, decorations.pathreplacing, decorations.shapes }
\usetikzlibrary{graphs}

\newcommand{\ot}{\otimes}
\definecolor{Maroon}{rgb}{0.6 0 0}
\definecolor{Prussian}{rgb}{0.05 0 0.6}
\definecolor{Emerald}{rgb}{0 0.5 0.1}

\newtheoremstyle{mytheorem}%
{10.0pt plus 2.0pt minus 2.0pt} 
{10.0pt plus 2.0pt minus 2.0pt} 
{\itshape} 
{} 
{\sc} 
{.} 
{ } 
{} 

\newtheoremstyle{mydefinition}%
{10.0pt plus 2.0pt minus 2.0pt} 
{10.0pt plus 2.0pt minus 2.0pt} 
{} 
{} 
{\sc} 
{.} 
{ } 
{} 

\newtheoremstyle{myexample}%
{10.0pt plus 2.0pt minus 2.0pt} 
{10.0pt plus 2.0pt minus 2.0pt} 
{\small} 
{} 
{\sc} 
{.} 
{ } 
{} 

\newtheoremstyle{myremark}%
{10.0pt plus 2.0pt minus 2.0pt} 
{10.0pt plus 2.0pt minus 2.0pt} 
{} 
{} 
{\itshape} 
{.} 
{ } 
{} 

\theoremstyle{mytheorem}
\newtheorem{theorem}{Theorem}[section]
\newtheorem{lemma}[theorem]{Lemma}

\newtheorem{corollary}[theorem]{Corollary}
\newtheorem{proposition}[theorem]{Proposition} 

\theoremstyle{myremark}
\newtheorem{remark}[theorem]{Remark}
\newtheorem{convention}[theorem]{Convention}
\newtheorem{notation}[theorem]{Notation}

\theoremstyle{mydefinition}
\newtheorem{definition}[theorem]{Definition}

\newtheorem{question}[theorem]{Question}

\theoremstyle{myexample}
\newtheorem{example}[theorem]{Example}

\newtheoremstyle{myzusatz}
 {10.0pt plus 2.0pt minus 2.0pt} 
{10.0pt plus 2.0pt minus 2.0pt} 
{\itshape} 
{} 
{\sc} 
{.} 
{ } 
{\thmname{#1}\thmnumber{ #2}\thmnote{ #3}}

\theoremstyle{myzusatz}

\definecolor{gray1}{gray}{0.8}
\definecolor{gray2}{gray}{0.6}
\definecolor{gray3}{gray}{0.4}
\definecolor{gray4}{gray}{0.2}

\DeclareMathOperator{\Aut}{\mathrm{Aut}}

\let\Aut\relax

\DeclareMathOperator{\Aut}{\mathrm{Aut}}

\DeclareMathOperator{\im}{\mathrm{im}}
\DeclareMathOperator{\Obj}{\mathrm{Obj}}

\DeclareMathOperator{\Free}{\mathrm{Free}}

\DeclareMathOperator{\dottriangleleft}{\triangleleft\kern -3.2pt \cdot}

\definecolor{Maroon}{rgb}{0.6 0 0}
\definecolor{Prussian}{rgb}{0.05 0 0.6}
\definecolor{Emerald}{rgb}{0 0.5 0.1}
\definecolor{gray1}{gray}{0.8}
\definecolor{gray2}{gray}{0.6}
\definecolor{gray3}{gray}{0.4}
\definecolor{gray4}{gray}{0.2}

\newcommand{\set}[1]{\lbrace #1 \rbrace}

\newcommand{\One}{\mathbbm{1}}

\newcommand{\id}[1][]{\mathrm{id}_{#1}}
\newcommand{\blank}{\raisebox{-2pt}{\text{---}}}

\newcommand{\inv}[1]{{#1}^{-1}}

\newcommand{\Cc}{{\mathscr{C}}}

\newcommand{\refl}[2]{{#1}^{(#2)}}

\newcommand{\brGho}{\mathrm{b}\kern -1pt\mathbb{G}}

\title[Reflections and Drinfeld twists for set-theoretic Yang--Baxter maps]{Reflections and Drinfeld twists\\ for set-theoretic Yang--Baxter maps}\author{Davide Ferri}
\begin{document}
\begin{abstract}
The Yang--Baxter equation (\textsc{ybe}) and the reflection equation (\textsc{re}) both come from mathematical physics, and they can be defined in any monoidal category. 

For cartesian monoidal categories, we prove that every solution to the \textsc{re} provides a Drinfeld twist for a solution of the \textsc{ybe}. As we observe, Drinfeld twists of solutions are relevant for the following reason: two solutions to the \textsc{ybe} (in any strict monoidal category) are related by a Drinfeld twist, if and only if they induce equivalent representations of the braid group.

In the category of sets, it is known that every solution is associated with a structure group, which is a braided group in the sense of Lu, Yan, and Zhu (2000). Using De Commer's notion of a braided action, we then define group reflections for a braided group. We prove that group reflections provide group Drinfeld twists in the sense of Ghobadi.  

Finally, we characterise when a reflection on a solution $(X,r)$ can be extended to a group reflection on its structure group $G(X,r)$.
\end{abstract}
\address{%
\parbox[b]{0.9\linewidth}{Università di Torino, Department of Mathematics `G.\@ Peano', via
 Carlo Alberto 10, 10123 Torino, Italy.\\
 Vrije Universiteit Brussel, Department of Mathematics and Data Science, Pleinlaan 2, 1050, Brussels, Belgium.\\ ORCID: 0000-0001-8421-496X}}
 \email{d.ferri@unito.it, Davide.Ferri@vub.be}
\allowdisplaybreaks
\maketitle
\tableofcontents
\section{Introduction}
The set-theoretic Yang--Baxter equation (\textsc{ybe}) has been object of intense investigation since it was proposed by Drinfeld \cite{drinfeld2006some}: originated as a tool to produce linear solutions to the quantum Yang--Baxter equation, it then branched and found relevance well beyond its original purpose. 

We define the \textsc{ybe} here in the broadest known setting. Given a monoidal category $(\Cc,\ot,\One)$ with monoidal product $\ot$ and monoidal unit $\One$, a solution to the \textsc{ybe} is an object $V$ in $\Cc$ with a morphism $r\colon V^{\ot 2}\to V^{\ot 2}$ satisfying
\begin{equation}\label{ybe}\tag{\sc ybe}r_{12} r_{23} r_{12} = r_{23} r_{12} r_{23}.\end{equation}
Here $r_{12} = r\ot\id[V]$ and $r_{23} = \id[V]\ot r$ are endomorphisms of $V^{\ot 3}$---and even if we are being sloppy on the associativity constraints, this sloppiness will remain safe throughout the paper. We say that the solution is \textit{invertible} if $r$ is an isomorphism, and \textit{involutive} if $r^2 = \id[V^{\ot 2}]$.  The \textsc{ybe} is also called the \emph{braid relation}, and a solution is also called a \textit{braided object} in $\Cc$. The origin of this name is explained in Figure \ref{fig:braid}. \begin{figure}[t]
\begin{tikzpicture}[x=0.6cm, y=0.4cm]
\draw[rounded corners=20, draw=white, double=black, line width=2pt] (0,0)--(4,4)--(4,6);
\draw[rounded corners=20, draw=white, double=black, line width=2pt] (2,0)--(0,3)--(2,6);
\draw[rounded corners=20, draw=white, double=black, line width=2pt] (0,6)--(4,2)--(4,0);
\node() at (6,3) {$=$};
\draw[rounded corners=20, draw=white, double=black, line width=2pt] (8,0)--(8,2)--(12,6);
\draw[rounded corners=20, draw=white, double=black, line width=2pt] (10,0)--(12,3)--(10,6);
\draw[rounded corners=20, draw=white, double=black, line width=2pt] (8,6)--(8,4)--(12,0);
\end{tikzpicture}
\caption{Pictorial representation of the braid relation. Each strand represents an element of $X$, and each crossing represents an application of $r$.}\label{fig:braid}
\end{figure}
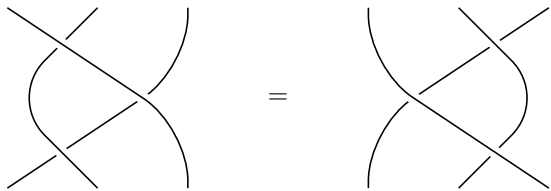All the braid-like diagrams, such as the one in Figure \ref{fig:braid}, will be read from top to bottom. 

To the present day, solutions to the \textsc{ybe} have mostly been investigated in the monoidal categories of sets ($\mathsf{Set}$), of $\Bbbk$-vector spaces ($\mathsf{Vec}_\Bbbk$), and of quivers over a set of vertices $\Lambda$ ($\mathsf{Quiv}_\Lambda$). The \textsc{ybe} in the category of vector spaces had already been introduced by Yang \cite{yang}, in the context of quantum integrable systems; and by Baxter \cite{baxter1972partition, baxter1985exactly}, independently, in statistical mechanics.  The term `set-theoretic \textsc{ybe}' refers to the \textsc{ybe} in $\mathsf{Set}$, which is the main topic of this paper.

The physical insight behind the \textsc{ybe} is roughly the following: the strands of Figure \ref{fig:braid} represent particles moving, with the arrow of time flowing from the top to the bottom of the diagram; and the Yang--Baxter map $r$ governs the interactions when the particles cross each other. The \textsc{ybe} on $r$, then, is related to the integrability of the system. 

But what if we introduce a non-holonomic constraint, forcing the particles to stay on the left-hand side of a `wall'? In the physical intuition, we now need an operator $k$ governing the interactions (`bouncings') between the particles and the wall. This operator should satisfy some relation with the `crossing' operator $r$: the relation is called the \textit{reflection equation}.

Again, we may define the reflection equation in full generality: given a solution $(V,r)$ to the \textsc{ybe} in a monoidal category $(\Cc,\ot ,\One)$, a morphism $k\colon V\to V$ satisfies the (right) reflection equation (\textsc{re}) with respect to $r$, if
\begin{equation}\label{re}\tag{\sc re} k_2 r k_2 r = r k_2 rk_2.\end{equation}
Here $k_2 = \id[V]\ot  k$. A graphical interpretation is given in Figure \ref{fig:re}. Of course, a notion of \textit{left} reflection equation $k_1 r k_1 r = rk_1 rk_1$, $k_1 = k\ot \id[V]$, may be given.

 The set-theoretic version of the \textsc{re} was given by Caudrelier, Crampé, and Zhang \cite{setTheoreticRE} (see also \cite{solitons,YBandReflDoikou,ReflectionVenSmoWes}). Hereafter, we concentrate on the case $\Cc = \mathsf{Set}$.
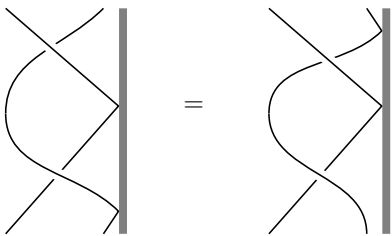
\begin{figure}
\centering
\begin{tikzpicture}
\draw[draw=white, double=black, line width=2pt] (-1.2, 3) to[out=-135, in=90] (-2.5, 1.6);
\draw[draw=white, double=black, line width=2pt] (-2.5,3) -- (-1, 1.7)--(-2.5, 0);
\draw[draw=white, double=black, line width=2pt] (-2.5, 1.6) to [out=-90] (-1,0.3) --(-1.2,0);
\draw[draw=white, double=black, line width=2pt] (2.3,3)--(2.5,2.7) to[out=-135, in=90] (1, 1.6);
\draw[draw=white, double=black, line width=2pt] (1,3) -- (2.5,1.7) --(1,0);
\draw[draw=white, double=black, line width=2pt] (1,1.6) to[out=-90, in=90] (2.3,0);
\draw[draw=gray, line width=3pt] (-0.94,0)--(-0.94,3);
\draw[draw=gray, line width=3pt] (2.56,0)--(2.56,3);
\node() at (0,1.7) {$=$};
\end{tikzpicture}
\caption{Pictorial representation of the reflection equation. Each crossing represents an application of $r$, and each bouncing on a lateral wall represents an application of $k$. By the \textsc{ybe} and the \textsc{re}, these diagrams can be considered up to homotopies that never drive the strands beyond the wall.}\label{fig:re}
\end{figure}

Lebed and Vendramin \cite{lebed2022reflection} made a crucial observation: the `bouncing sites' of the particles can be regarded as `virtual particles' themselves. The \textsc{re} in Figure \ref{fig:re}, then, describes a `crossing' between these two virtual particles. Notice that the old temporal dimension is now the spacial dimension where these virtual particles move; and the new arrow of time is instead the direction of the homotopy bringing the left-hand side of Figure \ref{fig:re} into the right-hand side. Lebed and Vendramin proved that this crossing of virtual particles is, in fact, a solution to the \textsc{ybe}, called the \textit{$k$-derived solution}; of which the \textit{derived solution} $r'$ is a special case.

In this paper, we prove that the $k$-derived solution is an instance of a \textit{Drinfeld twist} (we shall call this a \textit{reflection-twist}). Drinfeld twists of solutions have been defined by Kulish and Mudrov \cite{kulish2000twisting} (see also  \cite{ghobadi2021drinfeld}), and they are ways to deform a solution into a new one. Incidentally, we observe that Drinfeld twists can be defined for a solution in any monoidal category, and that two solutions are Drinfeld-twisted if and only if they induce equivalent representations of the braid group. 


A way to produce set-theoretic solutions is through \textit{braided groups} \cite{LYZ}; and conversely, every set-theoretic solution is associated with a certain braided group enjoying the universal property of a free object. Braided groups are groups $G$ equipped with a map $G\times G\to G\times G$ that behaves, in some way, like a \textit{constraint of abelianity} (for instance, the flip map $(a,b)\mapsto (b,a)$ is a braiding on $G$ if and only if $G$ is abelian). Such a `constraint of abelianity' is always a solution to the \textsc{ybe}.  

In \S\ref{sec:rebg}, we define a notion of \textit{group reflection} for braided groups, using De Commer's notion of braided actions \cite{de2019actions}. Interestingly, one of the axioms of a braided action will translate into \eqref{eq:weird}, which is a $k$-twisted involutivity. We then prove that group reflections are also Drinfeld twists, in the group-theoretic sense introduced by Ghobadi \cite{ghobadi2021drinfeld}.

Theorem \ref{thm:br-optimal} will tell us that, in the definition of group reflections, the first two axioms \eqref{eq:reflection-unity} and \eqref{eq:reflection-product} are `optimal' in order to make the twisted product into a group structure. The `twisted involutivity' \eqref{eq:weird} is not optimal in this sense---although it remains the most natural condition to require. An `optimal' condition is given in Theorem \ref{thm:br-optimal}. 

In \S\ref{sec:groupoid} we summarise what is known from \cite{ghobadi2021drinfeld} about the algebraic structure of group Drinfeld twists, and we add some new remarks, giving a complete cha\-rac\-te\-ri\-sa\-tion of `one-legged' group Drinfeld twists: this is a vast class of twists, which includes the ones constructed from reflections.

In the final part of the paper, we characterise when a reflection on a solution $(X,r)$ can be extended to a group reflection on the structure group $G(X,r)$. The same problem was considered in \cite[Example 2.4]{ghobadi2021drinfeld} for more general Drinfeld twists, and was therein left unsolved. We solve it here for the class of reflection-twists. 
\section{Drinfeld twists and representations of the braid group}\label{sec:twists}
The following is a way to twist Yang--Baxter maps into new solutions. The original definition of these twists in $\mathsf{Set}$ is due to Kulish and Mudrov \cite{kulish2000twisting}, but we adopt the definition as rephrased by Ghobadi \cite{ghobadi2021drinfeld}, and we formulate it for a braided object in a general monoidal category.
\begin{definition}[{\cite{kulish2000twisting}}]\label{def:twist}
Let $(V,r)$ be a braided object in $(\Cc,\ot,\One)$. Let $F\colon V^{\ot 2}\to V^{\ot 2}$ and $\Phi,\Psi\colon V^{\ot 3}\to V^{\ot 3}$ be isomorphisms such that
\begin{align}
\label{dt1}\tag{\sc dt1}&F_{12} \Psi = F_{23} \Phi;\\
\label{dt2}\tag{\sc dt2}&\Psi r_{12} = r_{12} \Psi;\\
\label{dt3}\tag{\sc dt3}&\Phi r_{23} = r_{23} \Phi.
\end{align}
Then, $r^F := F r F^{-1}$ is a solution on $V$, with exactly the same proof as \cite[Theorem 2.1]{ghobadi2021drinfeld}: we call it the solution \emph{twisted by $F$}. We call $F$, or more precisely the datum $(F,\Phi, \Psi)$, a \emph{Drinfeld twist} for $r$.
\end{definition}
Observe that $\Phi$ is uniquely determined by $F$ and $\Psi$, thus we shall write a Drinfeld twist $(F,\Phi,\Psi)$ as simply $(F,\Psi)$, implying that $\Phi = F_{23}^{-1}F_{12}\Psi$.

Of course, $F$ may \emph{a priori} be a Drinfeld twist for multiple choices of  $\Psi$, but this is hardly relevant, since the twisted solution $r^F$ only depends on $F$.

\begin{lemma}\label{lem:ghobadi}
Let $(F,\Phi,\Psi)$ be a Drinfeld twist for $r$, and $(G,\psi)$ be a Drinfeld twist for $r^F$. Then:
\begin{enumerate}
\item $(GF, F_{12}^{-1}\psi F_{12}\Psi)$ is a Drinfeld twist for $r$, and one has $r^{GF} = (r^F)^G$;
\item $(F^{-1},\inv{F_{12}}\Psi^{-1} F_{12})$ is a Drinfeld twist for $r^F$, and one has $(r^F)^{\inv{F}} = r$.
\end{enumerate}
\end{lemma} 
\begin{proof}Same as \cite[Theorem 2.2]{ghobadi2021drinfeld}.
\end{proof}
\begin{remark}
As we may expect, isomorphic solutions on the same object $V$ are Drinfeld-twist related. Indeed, if $(V,r)$ and $(V,s)$ are intertwined by an isomorphism $f\colon V\to V$, with $(f\ot f) r (f^{-1}\ot f^{-1}) = s$, then $F = f\ot f$ is a Drinfeld twist for $r$, with $\Psi = \id\ot \id \ot f$ and $\Phi = f\ot \id \ot \id$. This is an easy verification. For set-theoretic solutions, this was observed in \cite[\S2]{ghobadi2021drinfeld}.
\end{remark}
\begin{convention}\label{conv:extendtwist}When $(V,r) $ and $(W,s)$ are isomorphic solutions on different objects $V, W\in \Obj(\Cc)$, then we may identify $V$ and $W$ by means of the isomorphism. Thus we shall sometimes be sloppy on the terminology, and say that two solutions $(V,r)$ and $(W,s)$ are `Drinfeld-twist related' even when $(W,s)$ is a Drinfeld twist of some solution $(W,t)$ which is isomorphic to $(V,r)$, but $V$ and $W$ are not necessarily equal.\end{convention}

Doikou, Rybo{\l}owicz, and Stefanelli \cite{DoikouQuandlesAsPreLie} gave the following definition.
\begin{definition}[\cite{DoikouQuandlesAsPreLie}]\label{def:Dhom} A \textit{D-homomorphism} between solutions $(V,r)$ and $(W,s)$ is a morphism $F\colon V\ot V\to W\ot W$ such that $ Fr = sF$. This is called a \textit{D-isomorphism} if $F$ is an isomorphism in $\Cc$. \end{definition}
The fact that every solution in $\mathsf{Set}$ is D-isomorphic to its derived solution (and hence to a \textit{rack solution}) was already observed in \cite{DoikouQuandlesAsPreLie}. We are going to refine this result in Theorem \ref{thm:J-is-Drinfeld}.

If we extend the definition of Drinfeld twists to encompass solutions on dif\-fe\-rent objects, as in Convention \ref{conv:extendtwist}, then D-isomorphisms are clearly a weaker notion than Drinfeld twists. Corollary \ref{cor:DisoNotDtwists} will tell us that the notion is \textit{strictly} weaker.


\subsection{Drinfeld twists yield isomorphic representations}\label{sec:representations} We fix a monoidal ca\-te\-go\-ry $(\Cc,\ot, \One)$, which we assume to be strict, and an object $V$ in $\Cc$. If we denote by $V^\ot $ the full subcategory of tensor powers of $V$, then $V^\ot$ is monoidal with the restriction of $\ot$, and with $\One =  V^{\ot 0}$ as the unit.\footnote{If we do not assume the strictness, then $(V\ot V)\ot V$ and $V\ot (V\ot V)$ are possibly distinct albeit isomorphic objects, thus $V^\ot$ is only well-defined up to isomorphism. We can always consider a strictification of a monoidal category, where $V^\ot$ is well-defined.} 

If $r$ is a morphism $V^{\ot 2}\to V^{\ot 2}$, we use the well-established notation $r_{i,i+1}$ for the morphism $r$ applied on the $i$th and the $(i+1)$st leg of a tensor power. Notice that $r_{i, j}$ is generally ill-defined for $j\neq i+1$, unless the category $\Cc$ is braided. 
\begin{definition}
A \textit{representation} of a group $G$ in a category $\Cc$ is a group homomorphism $G\to \Aut(V)$ for some object $V$ in $\Cc$. 
\end{definition}
We recall that two group representations $\rho\colon G\to \Aut(V)$ and $\rho'\colon G\to \Aut(V')$ of the same group $G$ in the same category $\Cc$ are \textit{isomorphic} if there exists an isomorphism $a\colon V'\to V$ such that $\rho'(x) = a^{-1}\rho(x)a $ for all $x\in G$; i.e., if $\rho' = C_a \rho$, where $C_a$ denotes the conjugation by the morphism $a$.

\begin{notation}
We denote by $\mathbb{B}_n$ the braid group on $n$ strands \cite{garside1969braid}, and by $\sigma_{i, i+1}$ its generators, subject to the relations $\sigma_{i,i+1}\sigma_{i+1, i+2}\sigma_{i,i+1} = \sigma_{i+1, i+2}\sigma_{i,i+1}\sigma_{i+1, i+2}$ and $\sigma_{i, i+1}\sigma_{j,j+1} = \sigma_{j,j+1}\sigma_{i,i+1}$ if $|i-j|>1$.
\end{notation}

The following facts are very well known.
\begin{remark}
Having a braiding $r$ on the category $V^\ot $ is the same as having a solution of the \textsc{ybe} $r_{V,V}\colon V^{\ot 2}\to V^{\ot 2}$. Indeed, every braiding $r$ specialises to a solution $r_{V,V}$. Conversely, every solution $r_{V,V}$ can be extended to morphisms $r_{V^{\ot n}, V^{\ot m}}\colon V^{\ot n}\ot V^{\ot m}\to V^{\ot m}\ot V^{\ot n}$ in the obvious way, as depicted in Figure \ref{fig:extension-R}, and this is clearly a braiding on $V^\ot$. 

For the same reasons, let $\rho\colon \mathbb{B}_3\to \Aut(V^{\ot 3})$ be a representation of the braid group on $3$ strands, of the following special form:  $\rho$ sends $\sigma_{12}$ to $r\ot \id$ and $\sigma_{23}$ to $\id\ot r$, for some map $r\colon V^{\ot 2}\to V^{\ot 2}$. Then, this can be extended to a braiding on $V^\ot$, or equivalently to a representation of the braid groups $\mathbb{B}_n$ for all $n$, by setting $\sigma_{i, i+1}\mapsto r_{i, i+1}$. Conversely, every braiding on $V^{\ot}$ clearly induces a representation of $\mathbb{B}_3$ on $V^{\ot 3}$.\end{remark}
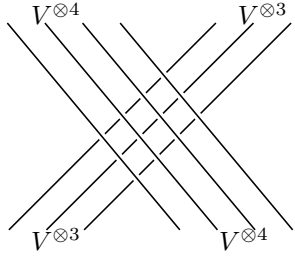
\begin{figure}[t]
\begin{center}
\begin{tikzpicture}[x=0.5cm, y=0.5cm]
\node (a1) at (-4,3) {};
\node (a2) at (-3,3) {};
\node (a3) at (-2,3) {};
\node (a4) at (-1,3) {};
\node (b1) at (2,3) {};
\node (b2) at (3,3) {};
\node (b3) at (4,3) {};
\node (B1) at (-4,-3) {};
\node (B2) at (-3,-3) {};
\node (B3) at (-2,-3) {};
\node (A1) at (1,-3) {};
\node (A2) at (2,-3) {};
\node (A3) at (3,-3) {};
\node (A4) at (4,-3) {};
\draw[draw=white, double=black, very thick] (b1) to (B1);
\draw[draw=white, double=black, very thick] (b2) to (B2);
\draw[draw=white, double=black, very thick] (b3) to (B3);
\draw[draw=white, double=black, very thick] (a1) to (A1);
\draw[draw=white, double=black, very thick] (a2) to (A2);
\draw[draw=white, double=black, very thick] (a3) to (A3);
\draw[draw=white, double=black, very thick] (a4) to (A4);
\node () at (-2.5,3) {$V^{\ot 4}$};
\node () at (2.5,-3) {$V^{\ot 4}$};
\node () at (3,3) {$V^{\ot 3}$};
\node () at (-2.5,-3) {$V^{\ot 3}$};
\end{tikzpicture}
\end{center}
\caption{Extending $r_{V,V}$ (here depicted as a crossing) to a map $r_{V^{\ot n}, V^{\ot m}}$ (in this case $n = 4$ and $m = 3$). Notice that, in the above diagram, only interactions of the form $r_{i, i+1}$ appear: thus the extension is well-defined in every monoidal category. }\label{fig:extension-R}
\end{figure}

\begin{proposition}\label{prop:iso_representations}
Let $(\Cc,\ot, \One)$ be a monoidal category,  $r\colon V^{\ot 2}\to V^{\ot 2}$ be a solution on an object $V$ of $\Cc$, and $F\colon V^{\ot 2}\to V^{\ot 2}$ be an isomorphism. Then the representations of $\mathbb{B}_3$ in $\Cc$
\[ \begin{cases}
\rho\colon \mathbb{B}_3\to \Aut(V^{\ot 3})\\
\sigma_{12}\mapsto r\ot \id\\
\sigma_{23}\mapsto \id\ot r
\end{cases} \qquad \begin{cases}
\rho'\colon \mathbb{B}_3\to \Aut(V^{\ot 3})\\
\sigma_{12}\mapsto FrF^{-1}\ot \id\\
\sigma_{23}\mapsto \id\ot FrF^{-1}
\end{cases}  \]
are isomorphic if and only if $(F, \Psi)$ is a Drinfeld twist for a suitable isomorphism $\Psi\colon V^{\ot 3}\to V^{\ot 3}$.
\end{proposition}
\begin{proof}
If $F$ is a Drinfeld twist for a suitable map $\Psi$, one easily checks that $\rho'= C_a\rho$ for $a = F_{12}\Psi$. Conversely, if $\rho' = C_a \rho$, set $\Psi = F_{12}^{-1}a$: then $(F,\Psi)$ is immediately seen to be a Drinfeld twist.
\end{proof}
As a consequence of the previous proof, the set of isomorphisms $\Psi$ that make $F$ into a Drinfeld twist is in bijection with the set of isomorphisms between the representations $\rho$ and $\rho'$.

\section{Reflections and Drinfeld twists for braided sets}
In the case $\Cc = \mathsf{Set}$, a solution $(X,r)$ is also called a \emph{braided set} in the literature, as well as here. We write
\[ r(a,b) =  (\sigma_a(b), \rho^{}_b(a)) = (a\rightharpoonup b, a\leftharpoonup b),\]
and we shall regularly switch between the two notations, according to what comes most in handy. In terms of the maps $\rightharpoonup$ and $\leftharpoonup$ (we hereafter say `in components'), the \textsc{ybe} reads:
	\begin{align}
		 &\label{ybe1}\tag{\sc ybe1}(a\rightharpoonup b )\rightharpoonup \left((a\leftharpoonup b)\rightharpoonup c\right) = a\rightharpoonup (b\rightharpoonup c);\\
		&\label{ybe2}\tag{\sc ybe2}(a\rightharpoonup b )\leftharpoonup \left( (a\leftharpoonup b)\rightharpoonup c\right) = \left( a\leftharpoonup (b\rightharpoonup c)\right)\rightharpoonup (b\leftharpoonup c);\\
		&\label{ybe3}\tag{\sc ybe3}(a\leftharpoonup b)\leftharpoonup c = \left( a\leftharpoonup (b\rightharpoonup c)\right) \leftharpoonup (b\leftharpoonup c).
	\end{align}
We denote by $M(X,r)$ and $G(X,r)$ the structure monoid and the structure group of $r$, respectively; see \cite{etingof1999set}. Given a map $k\colon X\to X$, we say that it is a (\textit{right}) \emph{$r$-reflection} (see \cite{lebed2022reflection}) if it satisfies the reflection equation \eqref{re} with respect to $r$. The \textsc{re} in components reads:
\begin{align}
&\label{re1}\tag{\sc re1}
(a\rightharpoonup b)\rightharpoonup k(a\leftharpoonup b) = (a\rightharpoonup k(b))\rightharpoonup k(a\leftharpoonup k(b));\\
&\label{re2}\tag{\sc re2} k\big((a\rightharpoonup b)\leftharpoonup k(a\leftharpoonup b) \big) = (a\rightharpoonup k(b))\leftharpoonup k(a\leftharpoonup k(b)).
\end{align}
\begin{notation}
For a non-degenerate solution $r\colon X\times X\to X\times X$, we denote $\inv{(a\rightharpoonup\blank)}$ simply by $\inv{a}\rightharpoonup\blank$, and we use an analogous notation for $\leftharpoonup$. Observe that this notation is consistent when $r$ is a braiding on a group (see Definition \ref{def:br_group}).
\end{notation}
Let $J^{k;n}$ be the \emph{generalised guitar map for $k$}, or \emph{$k$-guitar map}; see \cite{lebed2022reflection}. One has $\refl{r}{k}= Jr \inv{J}$, where $J:= J^k:= J^{k;2}$ is defined by $(a,b)\mapsto (a\leftharpoonup k(b), b)$. The \emph{$k$-twisted} or \emph{$k$-derived solution} $\refl{r}{k}$ takes the following form:
\begin{align*}\refl{r}{k}(a,b)&= (a\stackrel{(k)}{\rightharpoonup} b, a\stackrel{(k)}{\leftharpoonup}b) = (\refl{\sigma}{k}_a(b), \refl{\rho}{k}_b(a)) \\
&= \Big( \big((a\leftharpoonup\inv{k(b)})\rightharpoonup b\big)\leftharpoonup k\big((a\leftharpoonup\inv{k(b)})\leftharpoonup b\big),\; (a\leftharpoonup\inv{k(b)})\leftharpoonup b \Big)\\
& = \Big( \rho^{}_{k\rho^{}_b\rho^{-1}_{k(b)}(a)}\lambda_{\rho^{-1}_{k(b)}(a)}(b),\;  \rho^{}_b\rho^{-1}_{k(b)}(a)\Big).
\end{align*}
The twisted solution is also a Yang--Baxter map on $X$; see \cite{lebed2022reflection}. 
The operation of $k$-twisting generalises the (\textit{right}) \textit{derived solution} \[r' (a,b) := \big(( (a\leftharpoonup b^{-1})\rightharpoonup b)\leftharpoonup a, a\big),\] which is indeed a $k$-twist for the reflection $k = \id$.

\subsection{Reflections yield Drinfeld twists} We now prove that the $k$-guitar map $J:= J^{k;2}$ for a reflection $k$ is a Drinfeld twist, and hence $\refl{r}{k} = JrJ^{-1}$ is a Drinfeld-twisted solution.
\begin{theorem}\label{thm:J-is-Drinfeld}
Let $k$ be a reflection for $r$, with $r$ non-degenerate. Then, the associated guitar map $J = J^{k;2}$ is a set-theoretic Drinfeld twist.
\end{theorem}
\begin{proof}
Recall that $J(a,b) = (a\leftharpoonup k(b),b)$. Define
\begin{align*}
\Phi(a,b,c)&=\Bigg( \bigg(a\leftharpoonup \Big( (b\rightharpoonup k(c))\rightharpoonup k(b\leftharpoonup k(c))\Big)\bigg)\Bigg. \\ &\hspace{2em} \Bigg.\leftharpoonup \big( (b\rightharpoonup k(c))\leftharpoonup k(b\leftharpoonup k(c)) \big),\; b,\; c\Bigg)\\
\overset{\eqref{ybe3}}&{=} \Bigg( \big( a\leftharpoonup (b\rightharpoonup k(c))\big)\leftharpoonup k(b\leftharpoonup k(c)) ,\; b,\; c\Bigg),\end{align*}and
\begin{align*}
\Psi(a,b,c)&= \Big( a\leftharpoonup(b\rightharpoonup k(c)),\; b\leftharpoonup k(c),\; c \Big).
\end{align*}
Since $r$ is non-degenerate, it is immediate to prove that $J$, $\Phi$ and  $\Psi$ are bijective. 
It is an easy computation to verify \eqref{dt1}, \eqref{dt2}, and \eqref{dt3}.
\end{proof}
\begin{remark}
Observe that both sides of \eqref{dt1} equal the map $J^{k;3}$ defined in \cite{lebed2022reflection}.
\end{remark}
\begin{remark}
Theorem \ref{thm:J-is-Drinfeld} was stated for \textit{right} reflections. With a dual proof, one can show that \textit{left} reflections also induce Drinfeld twists.
\end{remark}
Since Drinfeld twists can be composed and inverted, it is natural to investigate compositions and inverses of reflections. 
\begin{lemma}\label{lem:explicit}
Let $r$ be right non-degenerate $r(a,b) = (\sigma_a(b), \rho^{}_b(a))$, let $k$ be a reflection for $r$, and let $h$ be a reflection for $\refl{r}{k}$. Then, the map $\refl{\left( \refl{r}{k}\right)}{h}$ acts as follows:
\[ (a, b)\mapsto (u,v)  \]
where 
\begin{align*}
u&=\rho^{}_{h \rho^{}_b R_b^{-1} (a) }\left( \rho^{}_{kh\rho^{}_b R_b^{-1}(a)}\right)^{-1} \rho^{}_{k\rho^{}_b R_b^{-1}(a)}\lambda_{R_b^{-1}(a)}(b) ,\\
v&=\rho^{}_b R_b^{-1}(a),\\
R_b&=  \rho^{}_{h(b)}\inv{\rho^{}_{kh(b)}}\rho^{}_{k(b)}.
\end{align*}
\end{lemma}
\begin{proof}
By definition of $\refl{r}{k}$, one has $\refl{\refl{\rho^{}_b}{k}}{h} =\rho^{}_b  \inv{\rho^{}_{k(b)}} \inv{\rho^{}_{kh(b)}} \rho^{}_{h(b)}$. Thus $\refl{\left( \refl{r}{k}\right)}{h}$ maps $\left(\rho^{}_{h(b)}\inv{\rho^{}_{kh(b)}}(\alpha),\; b\right)$ into $\left( \alpha' , \; \rho^{}_{h(\alpha')}\inv{\rho^{}_{kh(\alpha')}}(b') \right)$, where
$$\alpha' = \rho^{}_{b}\inv{\rho^{}_{k(b)}}(\alpha),\quad b'= \rho^{}_{k(\alpha')}\lambda_{\inv{\rho^{}_{k(b)}}(\alpha)}(b).$$
The conclusion follows by substituting $ a: =\rho^{}_{h(b)}\inv{\rho^{}_{kh(b)}}(\alpha) $. This substitution is bijective because $r$ is right non-degenerate.
\end{proof}
After contemplating the expression in Lemma \ref{lem:explicit}, the following is immediate.
\begin{corollary}\label{cor:composition}
Suppose that $\ell$ is a reflection for $r$, where $r\colon (a,b) \mapsto (\sigma_a(b), \rho^{}_b(a))$ is right non-degenerate. Then, $\refl{\left( \refl{r}{k}\right)}{h}=\refl{r}{\ell}$ holds if and only if the following holds for all $b$:
\begin{equation}\label{eq:composition_sufficient_necessary_condition}\rho^{}_{\ell(b)} = \rho^{}_{h(b)}\inv{\rho^{}_{hk(b)}} \rho^{}_{k(b)}.\end{equation}
In particular, $\refl{(\refl{r}{k})}{h}= r$ if and only if $\rho^{}_{h(b)}\inv{\rho^{}_{hk(b)}} \rho^{}_{k(b)} = \id $ holds for all $b$. \end{corollary}
In Corollary \ref{cor:composition}, condition \eqref{eq:composition_sufficient_necessary_condition} does not ensure that $\ell$ is a reflection. We shall indeed present co\-un\-ter\-ex\-amples in Example \ref{ex:ell_not_refl}.  Thus we leave the following (seemingly hard) question open.
\begin{question}\label{quest:composing_guitar_maps} By Theorem \ref{thm:J-is-Drinfeld} and Lemma \ref{lem:ghobadi}, compositions and inverses of guitar maps are Drinfeld twists. What compositions and what inverses of guitar maps are still guitar maps for some reflection?\end{question}

In the proof of Theorem \ref{thm:J-is-Drinfeld}, we crucially used the fact that $r$ can be written in components $\rightharpoonup,\leftharpoonup$, which is a property of $\mathsf{Set}$ that cannot be generalised to any monoidal category. The existence of the two components descends from the fact that $\mathsf{Set}$ is cartesian; i.e., that the tensor product $\times$ in $\mathsf{Set}$ also enjoys the universal property of a product.
\begin{remark}
	Let $(V,r)$ be a solution in a cartesian monoidal category $(\Cc, \times, \One)$. Let $k$ be a reflection for $r$. Then $r\colon X\times X\to X\times X$ is uniquely determined by maps $\rightharpoonup,\leftharpoonup\colon X\times X\to X$ by the universal property of the product. The notions of left and right non-degeneracy are well defined, and it is easy to verify that Theorem \ref{thm:J-is-Drinfeld} holds \textit{verbatim} with the same proof.
\end{remark}
\subsection{Permutation solutions}
%
Let $\mathfrak{S}_X$ be the set of permutations of $X$. For $\lambda,\rho$ two commuting permutations in $\mathfrak{S}_X$, it is known that $r\colon (a,b)\mapsto (\lambda(b),\rho(a))$ is a solution, called a \textit{permutation solution}.
\begin{lemma}\label{lem:perm}
Let $r\colon (a,b)\mapsto (\lambda(b),\rho(a))$ be a permutation solution. Then its class modulo Drinfeld twists depends exclusively on the product $\rho\lambda$.
\end{lemma}
\begin{proof}
Clearly $r$ is a Drinfeld twist of the derived solution $r'$. It suffices to observe that $r'(a,b) = (\rho\lambda(b), a)$.
\end{proof}
\begin{remark}
Let $r\colon (a,b)\mapsto (\lambda(b), a)$ be a permutation solution. Then $k$ is a reflection if and only if $k\lambda = \lambda k$: this is immediate from \eqref{re1} and \eqref{re2}. Observe that such a map $k$ always exists (e.g.\@ $k =\id$). The corresponding Drinfeld twist is $F(a,b) = (a, \lambda(b))$, and the twisted solution is $\refl{r}{k}  = r'\colon (a,b)\mapsto (a, \lambda(b))$.

Dually, every permutation solution $r\colon (a,b)\mapsto (b, \rho(a))$ can be twisted into $\refl{r}{k}\colon (a,b)\mapsto (\rho(a), b)$ by means of a \textit{left reflection} $k$, satisfying $k\rho = \rho k$.
\end{remark}

Let $r\colon (a,b)\mapsto (\lambda(b), a)$ be a permutation solution on $X$. For $p,q\in \mathfrak{S}_X$, the map $F\colon (a,b)\mapsto (p(a), q(b))$ is bijective, and 
\[ FrF^{-1}\colon (a,b)\mapsto (p\lambda q^{-1}(b), qp^{-1}(a)) \]
is a solution if and only if $p\lambda p^{-1} = q\lambda q^{-1}$. In that case, $F$ is obviously a $D$-isomorphism, and we wonder whether it is also a Drinfeld twist.
\begin{lemma}\label{lem:permutation-Dtwist}
	In the above hypotheses, $F$ is a Drinfeld twist if and only if $\lambda = \id$.
\end{lemma}
\begin{proof}
	We search for bijective maps $\Phi,\Psi\colon X^3\to X^3$ that make $F$ into a Drinfeld twist. Since \eqref{dt1} implies $\Phi = F_{12}\Psi F_{23}^{-1}$, it suffices to find a bijective 
	\[  \Psi\colon (a,b,c)\mapsto (\Psi_1(a,b,c), \Psi_2(a,b,c), \Psi_3(a,b,c))\] such that
	\begin{align}
		\label{dt2.1} &\Psi_1(\lambda(b), a, c) = \lambda \Psi_2(a,b,c),\\
		\label{dt2.2} &\Psi_2(\lambda(b), a, c) = \Psi_1(a,b,c),\\
		\label{dt2.3}& \Psi_3(\lambda(b), a,c) = \Psi_3(a,b,c),\\
		\label{dt3.1} & p\Psi_1(a, p^{-1}\lambda(c), q^{-1}(b)) = p\Psi_1(a, p^{-1}(b), q^{-1}(c)),\\
		\label{dt3.2} &q\Psi_2(a, p^{-1}\lambda(c), q^{-1}(b)) = \lambda \Psi_3(a, p^{-1}(b), q^{-1}(c)),\\
		\label{dt3.3} &\Psi_3(a, p^{-1}\lambda(c), q^{-1}(b)) = q\Psi_2(a, p^{-1}(b), q^{-1}(c)),
	\end{align}
	where \eqref{dt2.1}--\eqref{dt2.3} translate \eqref{dt2}, and \eqref{dt3.1}--\eqref{dt3.3} translate \eqref{dt3}. From \eqref{dt2.1} and \eqref{dt3.3} respectively, we have
	\begin{align*}
		\Psi_2(a,b,c) &= \lambda^{-1}\Psi_1(\lambda(b), a, c),\\
		\Psi_3(a,b,c) &= q\Psi_2(a, p^{-1}q(c), q^{-1}\lambda^{-1}p(b))\\
		&= q\lambda^{-1}\Psi_1(\lambda p^{-1}q(c), a, q^{-1}\lambda^{-1}p(b)).
	\end{align*}
	Substituting these expressions in \eqref{dt2.2} and \eqref{dt2.3} respectively, we get
	\begin{align*}
		\lambda\Psi_1(\lambda p^{-1}q(c), a, q^{-1}\lambda^{-1}p(b))
		\overset{\eqref{dt2.2}}&{=}\Psi_1(\lambda(a), \lambda p^{-1}q(c), q^{-1}\lambda^{-1}p(b))\\
		\overset{\eqref{dt2.3}}&{=} \Psi_1(\lambda p^{-1}q(c), a, q^{-1}\lambda^{-1}p(b)),
	\end{align*}
	and since $\lambda$, $p$ and $q$ are bijections, this simply means that $\im(\Psi_1)$ is contained in the fixed points of $\lambda$. But $\Psi\colon X^3\to X^3$ is bijective, thus $\im(\Psi_1)$ must be the whole of $X$, whence $\lambda = \id$.
\end{proof}
\begin{corollary}\label{cor:DisoNotDtwists}
	There are D-isomorphisms that are not Drinfeld twists.
\end{corollary}

\section{Reflections and Drinfeld twists for braided groups}\label{sec:rebg}
We recall the classical definition of a braided group, as it appeared in Lu, Yan, and Zhu \cite{LYZ}.
\begin{definition}\label{def:br_group}
Let $G$ be a group. A \emph{braiding} on $G$ is a bijective map $r\colon G\times G\to G\times G$, satisfying
\begin{align}
\label{eq:a1-1a}\tag{\sc bg1}& r(a,1) = (1,a),\\
\label{eq:1b-b1}\tag{\sc bg2}&r(1,b) = (b,1),\\ 
\label{eq:hex:rm23}&\tag{\sc bg3}
r m_{23} = m_{12}r_{23}r_{12},\\
\label{eq:hex:rm12}\tag{\sc bg4}&r m_{12}=  m_{23}r_{12}r_{23},\\
\label{eq:mult}&\tag{\sc bg5}
mr = m,\end{align}
where $m\colon G\times G\to G$ is the multiplication of $G$. 
\end{definition}
Equivalenty (see \cite{LYZ}), a braiding on $G$ is provided by a pair of actions $\rightharpoonup,\leftharpoonup\colon G\times G\to  G$, left and right respectively, which are compatible in the following sense:
\[(a\rightharpoonup b)(a\leftharpoonup b) = ab.\]
We switch between the two equivalent settings by imposing $r(a,b) = (a\rightharpoonup b, a\leftharpoonup b)$. A morphism of braided groups is a group homomorphism that intertwines the respective braidings. We denote by $\mathsf{BrGp}$ the category of braided groups.

It is well known that braidings on groups are also bijective non-degenerate solutions to the \textsc{ybe} \cite{LYZ}.
\begin{remark}
	Let $(a,b)\mapsto (a\rightharpoondown b, a\leftharpoondown b)$ denote the inverse of $r$. In a braided group, one easily has (see \cite{LYZ})
	\[  a\rightharpoondown b = (b^{-1}\leftharpoonup a^{-1})^{-1},\quad a\leftharpoondown b = (b^{-1}\rightharpoonup a^{-1})^{-1}.\]
	It is known that $r^{-1}$ is also a braiding.
\end{remark}

The following equivalent formulation of braided groups was given by Guarnieri and Vendramin \cite{guarnieri2017skew}, relieving an abelianity hypothesis from Rump \cite{rump2007braces}.
\begin{definition}[\cite{guarnieri2017skew,rump2007braces}]
A \textit{skew brace} $(G,+,\cdot)$ is the datum of a set $G$ with two group operations (we denote the first one additively, without implying it to be abelian, and the second one multiplicatively), satisfying the compatibility
\[ a(b+c) = ab - a+ ac. \]
A skew brace is called a \textit{brace} if $(G,+)$ is abelian.
\end{definition}
\begin{proposition}[\cite{guarnieri2017skew,rump2007braces}]
Let $(G,\cdot)$ be a group. The following data are equivalent:
\begin{enumerate}
\item a braiding $r$;
\item a second group operation $+$ such that $(G,+,\cdot)$ is a skew brace.
\end{enumerate}
The correspondence is as follows: given a braiding $r$, we define $a+b = a(a^{-1}\rightharpoonup b)$; and given the group operation $+$, we define $a\rightharpoonup b = -a+ab$, $a\leftharpoonup b = (a\rightharpoonup b)^{-1}ab$.

Through this correspondence, involutive braidings correspond to braces. 
\end{proposition}
For every group $(G,\cdot)$, the datum $(G,\cdot, \cdot)$ is a skew brace, which is called the \textit{trivial skew brace} on $G$. This corresponds to the braiding $(a,b)\mapsto (b, b^{-1}ab)$. 

\subsection{Group reflections} Following \cite{de2019actions}, we now give the definition of reflections for a braided group. This is a stronger notion than just a reflection for the associated set-theoretic solution.
\begin{definition}
A map $k\colon G\to G$ is said to be a  \emph{right group reflection} for $(G,r)$ if the following conditions are satisfied:
\begin{align}
\label{eq:reflection-unity}\tag{\sc bre1}&k(1) = 1,\\\label{eq:reflection-product}\tag{\sc bre2}&k(ab) = (a\rightharpoonup k(b))\; k(a\leftharpoonup k(b)),\\
\label{eq:weird}\tag{\sc bre3}&
k(a) = (a\rightharpoonup b)\rightharpoonup k(a\leftharpoonup b)\text{ for all }b.
\end{align}
A map $k\colon G\to G$ is a \textit{left group reflection} if it satisfies
\begin{align}
	\label{eq:right_reflection-unity}\tag{\sc bre1$'$}&k(1) = 1,\\\label{eq:right_reflection-product}\tag{\sc bre2$'$}&k(ab) = k(k(a)\rightharpoonup b)\; (k(a)\leftharpoonup b),\\
	\label{eq:right_weird}\tag{\sc bre3$'$}&
	k(b) = k(a\rightharpoonup b)\leftharpoonup (a\leftharpoonup b)\text{ for all }a.
\end{align}
When we say `group reflection', we always mean a right group reflection.
\end{definition}
Notice that the notion of group reflection is the same as De Commer's notion of a  braided action \cite[Definition 6.1]{de2019actions} with $G$ acting on a singleton. 
\begin{lemma}
If $k$ is a (right, resp.\@ left) group reflection for $r$, then it is also a (right, resp.\@ left) reflection.
\end{lemma}
\begin{proof}
For right reflections, it is a special case of \cite[Lemma 6.4]{de2019actions}. For left reflections, the proof is symmetrical.
\end{proof}
An interpretation of \eqref{eq:reflection-unity} and \eqref{eq:reflection-product} will be given in \S\ref{sec:extending}, after introducing a suitable `ribbon notation'. Condition \eqref{eq:weird} can be interpreted as a `twisted involutivity': indeed, when $k= \id$, \eqref{eq:weird} is equivalent to the involutivity of $r$.

Recall that, for every braided set $(X,r)$, the map $k = \id$ and the constant map $k\colon x\mapsto c$ are reflections. On a braided group, $k =\id$ is a group reflection if and only if the braiding is involutive, while on the other hand, the constant map $k\colon x\mapsto 1$ is always a group reflection.
\begin{remark}
Given $(a,b)\in G\times G$, let $(a_1, b_1) = k_2 rk_2r(a,b)$ and $(a_2, b_2) = rk_2 rk_2(a,b)$ be the two sides of the \textsc{re} applied to $(a,b)$, and suppose that $k\colon G\to G$ satisfies \eqref{eq:reflection-unity} and \eqref{eq:reflection-product}, but not necessarily \eqref{eq:weird}. With the same proof as \cite[Lemma 6.4]{de2019actions}, one can still prove that $a_1b_1 = a_2 b_2$.
\end{remark}
\begin{remark}\label{rem:k|_X}
Let $(G,r)$ be a braided group with a set of generators $X$. Then \eqref{eq:reflection-unity} and \eqref{eq:reflection-product} imply that a group reflection $k$ is uniquely determined by its value on the generators. 

Suppose now that $r$ restricts to a map $X\times X\to X\times X$, and that $k$ restricts to a map $X\to X$. If \eqref{eq:weird} holds in $G$, then it holds in particular for the restriction $k|_X$ of $k$ to $X$. Conversely, if $k|_X$ satisfies \eqref{eq:weird} for all $a,b\in X$, then using \eqref{eq:reflection-unity} and \eqref{eq:reflection-product} we can extend $k|_X$ to a map $k$ satisfying \eqref{eq:weird} for all $a,b\in G$. Condition \eqref{eq:weird} on $k$ can be proved inductively by iterated applications of \eqref{eq:weird} on the single strands; see Figure \ref{fig:proof_weird_0}.
\end{remark}
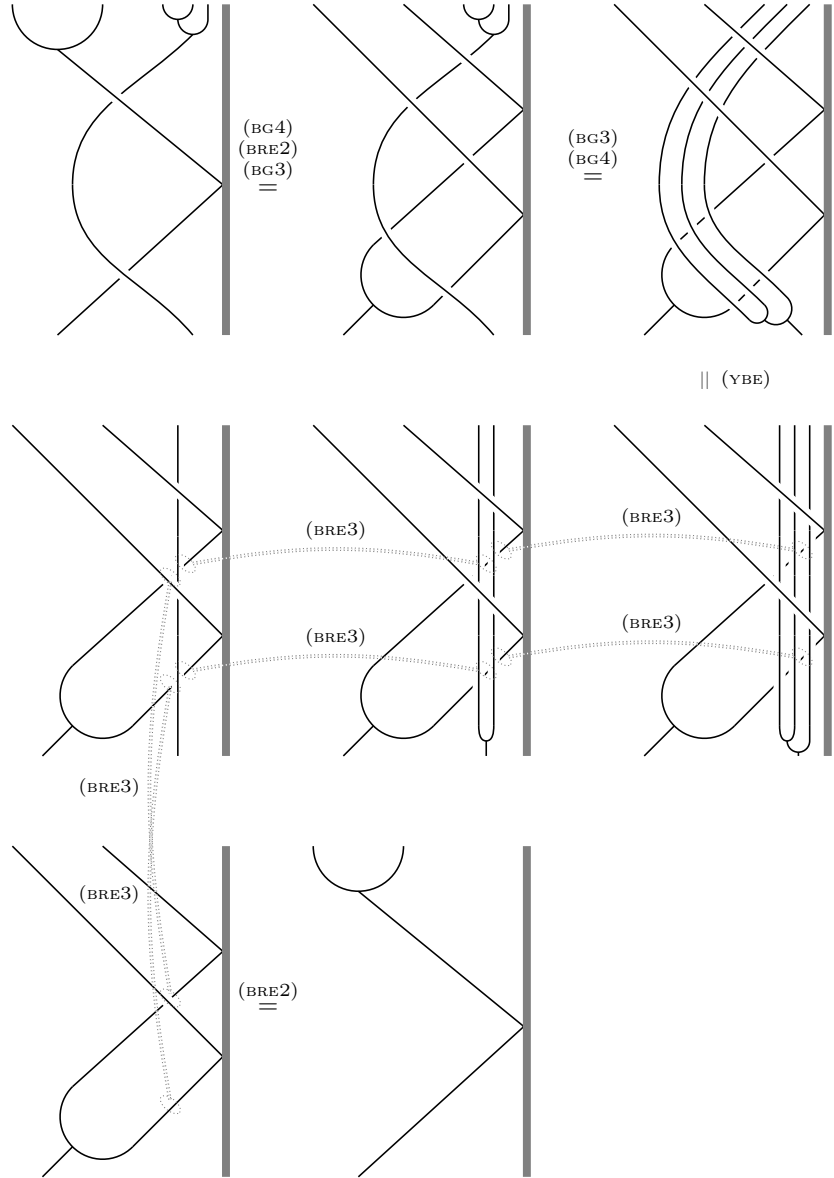
\begin{figure}[t]
\begin{center}
\begin{tikzpicture}[x=0.4cm, y=0.4cm]
\draw[draw=white, double=black, line width=2pt](6,10) to[out=-135,in=90](2,5);
\draw[draw=white, double=black, line width=2pt] (1.5,9.5)--(7,5)--(1.5,0);
\draw[line width=0.6pt] (0,11) to[out=-90, in=180](1.5,9.5) to[out=0, in=-90] (3,11);
\draw[draw=white, double=black, line width=2pt] (2,5) to[out=-90, in=130](6,0);
\draw[line width=0.6pt] (5,11) to[out=-90, in=180](5.5,10.5) to[out=0, in=-90] (6,11);
\draw[line width=0.6pt] (5.5,10.5) to[out=-90, in=180](6,10)to[out=0, in=-90] (6.5,10.5) to[out=90, in=-90] (6.5,11);
\draw[draw=gray, line width=3pt] (7.1,0) to (7.1,11);
%
\draw[draw=white, double=black, line width=2pt](16,10) to[out=-135,in=90](12,5);
\draw[draw=white, double=black, line width=2pt] (13,11)--(17,7.5)--(12,3);
\draw[draw=white, double=black, line width=2pt] (10,11)--(17,4)--(14,1);
\draw[draw=white, double=black, line width=2pt](12,1)--(11,0);
\draw[draw=white, double=black, line width=2pt] (12,5) to[out=-90, in=130](16,0);
\draw[line width=0.6pt] (12,3) to[out=-135, in=135](12,1) to[out=-45, in=-135] (14,1);
\draw[line width=0.6pt] (15,11) to[out=-90, in=180](15.5,10.5) to[out=0, in=-90] (16,11);
\draw[line width=0.6pt] (15.5,10.5) to[out=-90, in=180](16,10)to[out=0, in=-90] (16.5,10.5) to[out=90, in=-90] (16.5,11);
\draw[draw=gray, line width=3pt] (17.1,0) to (17.1,11);
%
\draw[draw=white, double=black, line width=2pt](25,11) to[out=-135,in=90](21.5,5);
\draw[draw=white, double=black, line width=2pt](25.75,11) to[out=-135,in=90](22.25,5);
\draw[draw=white, double=black, line width=2pt](26.5,11) to[out=-135,in=90](23,5);
\draw[draw=white, double=black, line width=2pt] (23,11)--(27,7.5)--(22,3);
\draw[draw=white, double=black, line width=2pt] (20,11)--(27,4)--(24,1);
\draw[draw=white, double=black, line width=2pt](22,1)--(21,0);
\draw[line width=0.6pt] (22,3) to[out=-135, in=135](22,1) to[out=-45, in=-135] (24,1);
\draw[draw=white, double=black, line width=2pt] (21.5,5) to[out=-90, in=135](24.5,0.5);
\draw[draw=white, double=black, line width=2pt] (22.25,5) to[out=-90, in=135](25,1);
\draw[draw=white, double=black, line width=2pt] (23,5) to[out=-90, in=135](25.5,1.5);
\draw[line width=0.6pt] (24.5,0.5) to[out=-45, in=-135](25,0.5) to[out=45, in=-45] (25,1);
\draw[line width=0.6pt] (25,0.5) to[out=-45, in=-135](25.75,0.5)to[out=45, in=-45] (25.75,1.25) to[out=135, in=-45] (25.5,1.5);
\draw[line width=0.6pt] (25.75,0.5)--(26.25,0);
\draw[draw=gray, line width=3pt] (27.1,0) to (27.1,11);
%
\draw[draw=white, double=black, line width=2pt](25.5,-3)--(25.5,-6.7);
\draw[draw=white, double=black, line width=2pt](26,-3)--(26,-6.5);
\draw[draw=white, double=black, line width=2pt](26.5,-3)--(26.5,-6.5);
\draw[draw=white, double=black, line width=2pt] (25.5,-8.2)--(25.5,-10);
\draw[draw=white, double=black, line width=2pt] (26,-8)--(26,-10);
\draw[draw=white, double=black, line width=2pt] (26.5,-8)--(26.5,-10);
\draw[draw=white, double=black, line width=2pt] (23,-3)--(27,-6.5)--(22,-11);
\draw[draw=white, double=black, line width=2pt] (20,-3)--(27,-10)--(24,-13);
\draw[draw=white, double=black, line width=2pt](22,-13)--(21,-14);
\draw[line width=0.6pt] (22,-11) to[out=-135, in=135](22,-13) to[out=-45, in=-135] (24,-13);
\draw[draw=white, double=black, line width=2pt] (25.5,-6.7)--(25.5,-8.2);
\draw[draw=white, double=black, line width=2pt] (26,-6.5)--(26,-8);
\draw[draw=white, double=black, line width=2pt] (26.5,-6.5)--(26.5,-8);
\draw[draw=white, double=black, line width=2pt] (25.5,-10)--(25.5,-13);
\draw[draw=white, double=black, line width=2pt] (26,-10)--(26,-13);
\draw[draw=white, double=black, line width=2pt] (26.5,-10)--(26.5,-13);
\draw[line width=0.6pt] (25.5,-13) to[out=-90, in=180](25.75,-13.5) to[out=0, in=-90] (26,-13);
\draw[line width=0.6pt] (25.75,-13.5) to[out=-90, in=180](26.125,-13.875) to[out=0, in=-90] (26.5,-13.5)to[out=90, in=-90](26.5,-13);
\draw[line width=0.6pt] (26.125,-13.875)--(26.125,-14);
\draw[draw=gray, line width=3pt] (27.1,-14) to (27.1,-3);
%
\draw[draw=white, double=black, line width=2pt](15.5,-3)--(15.5,-6.7);
\draw[draw=white, double=black, line width=2pt](16,-3)--(16,-6.5);
\draw[draw=white, double=black, line width=2pt] (15.5,-8.2)--(15.5,-10);
\draw[draw=white, double=black, line width=2pt] (16,-8)--(16,-10);
\draw[draw=white, double=black, line width=2pt] (13,-3)--(17,-6.5)--(12,-11);
\draw[draw=white, double=black, line width=2pt] (10,-3)--(17,-10)--(14,-13);
\draw[draw=white, double=black, line width=2pt](12,-13)--(11,-14);
\draw[line width=0.6pt] (12,-11) to[out=-135, in=135](12,-13) to[out=-45, in=-135] (14,-13);
\draw[draw=white, double=black, line width=2pt] (15.5,-6.7)--(15.5,-8.2);
\draw[draw=white, double=black, line width=2pt] (16,-6.5)--(16,-8);
\draw[draw=white, double=black, line width=2pt] (15.5,-10)--(15.5,-13);
\draw[draw=white, double=black, line width=2pt] (16,-10)--(16,-13);
\draw[line width=0.6pt] (15.5,-13) to[out=-90, in=180](15.75,-13.5) to[out=0, in=-90] (16,-13);
\draw[line width=0.6pt] (15.75,-13.5)--(15.75, -14);
\draw[draw=gray, line width=3pt] (17.1,-14) to (17.1,-3);
%
\draw[draw=white, double=black, line width=2pt](5.5,-3)--(5.5,-6.7);
\draw[draw=white, double=black, line width=2pt] (5.5,-8.2)--(5.5,-10);
\draw[draw=white, double=black, line width=2pt] (3,-3)--(7,-6.5)--(2,-11);
\draw[draw=white, double=black, line width=2pt] (0,-3)--(7,-10)--(4,-13);
\draw[draw=white, double=black, line width=2pt](2,-13)--(1,-14);
\draw[line width=0.6pt] (2,-11) to[out=-135, in=135](2,-13) to[out=-45, in=-135] (4,-13);
\draw[draw=white, double=black, line width=2pt] (5.5,-6.7)--(5.5,-8.2);
\draw[draw=white, double=black, line width=2pt] (5.5,-10)--(5.5,-14);
\draw[draw=gray, line width=3pt] (7.1,-14) to (7.1,-3);
%
\draw[draw=white, double=black, line width=2pt] (3,-17)--(7,-20.5)--(2,-25);
\draw[draw=white, double=black, line width=2pt] (0,-17)--(7,-24)--(4,-27);
\draw[draw=white, double=black, line width=2pt](2,-27)--(1,-28);
\draw[line width=0.6pt] (2,-25) to[out=-135, in=135](2,-27) to[out=-45, in=-135] (4,-27);
\draw[draw=gray, line width=3pt] (7.1,-28) to (7.1,-17);
%
\draw[draw=white, double=black, line width=2pt] (11.5,-18.5)--(17,-23)--(11.5,-28);
\draw[line width=0.6pt] (10,-17) to[out=-90, in=180](11.5,-18.5) to[out=0, in=-90] (13,-17);
\draw[draw=gray, line width=3pt] (17.1,-28) to (17.1,-17);
%
\draw[densely dotted, draw=gray, draw opacity=0.8] (26.25,-7.15) circle [x radius=0.4, y radius=0.2, rotate=-40];
\draw[densely dotted, draw=gray, draw opacity=0.8] (16.25,-7.15) circle [x radius=0.4, y radius=0.2, rotate=-40];
\draw[double, densely dotted, draw=gray, draw opacity=0.8] (16.5,-7.15) to[bend left=10] node [above]{\scriptsize\eqref{eq:weird}}(26,-7.15);
\draw[densely dotted, draw=gray, draw opacity=0.8] (26.25,-10.7) circle [x radius=0.4, y radius=0.2, rotate=-40];
\draw[densely dotted, draw=gray, draw opacity=0.8] (16.25,-10.7) circle [x radius=0.4, y radius=0.2, rotate=-40];
\draw[double, densely dotted, draw=gray, draw opacity=0.8] (16.5,-10.7) to[bend left=10] node [above]{\scriptsize\eqref{eq:weird}}(26,-10.7);
\draw[densely dotted, draw=gray, draw opacity=0.8] (15.75,-7.6) circle [x radius=0.4, y radius=0.2, rotate=-40];
\draw[densely dotted, draw=gray, draw opacity=0.8] (5.75,-7.6) circle [x radius=0.4, y radius=0.2, rotate=-40];
\draw[double, densely dotted, draw=gray, draw opacity=0.8] (6,-7.6) to[bend left=10] node[above]{\scriptsize\eqref{eq:weird}}(15.5,-7.6);
\draw[densely dotted, draw=gray, draw opacity=0.8] (15.75,-11.15) circle [x radius=0.4, y radius=0.2, rotate=-40];
\draw[densely dotted, draw=gray, draw opacity=0.8] (5.75,-11.15) circle [x radius=0.4, y radius=0.2, rotate=-40];
\draw[double, densely dotted, draw=gray, draw opacity=0.8] (6,-11.15) to[bend left=10] node[above]{\scriptsize\eqref{eq:weird}}(15.5,-11.15);
\draw[densely dotted, draw=gray, draw opacity=0.8] (5.25,-8.05) circle [x radius=0.4, y radius=0.2, rotate=-40];
\draw[densely dotted, draw=gray, draw opacity=0.8] (5.25,-22.05) circle [x radius=0.4, y radius=0.2, rotate=-40];
\draw[double, densely dotted, draw=gray, draw opacity=0.8] (5.25,-8.2) to[bend right=10] node[left]{\scriptsize\eqref{eq:weird}}(5.25,-21.90);
\draw[densely dotted, draw=gray, draw opacity=0.8] (5.25,-11.6) circle [x radius=0.4, y radius=0.2, rotate=-40];
\draw[densely dotted, draw=gray, draw opacity=0.8] (5.25,-25.6) circle [x radius=0.4, y radius=0.2, rotate=-40];
\draw[double, densely dotted, draw=gray, draw opacity=0.8] (5.25,-11.7) to[bend right=10] node[left]{\scriptsize\eqref{eq:weird}}(5.25,-25.5);
\node() at (8.5,-22) {$\overset{\eqref{eq:reflection-product}}{=}$};
\node() at (8.5,6) {$\overset{\scriptsize\begin{matrix}\eqref{eq:hex:rm12}\\ \eqref{eq:reflection-product}\\ \eqref{eq:hex:rm23}\end{matrix}}{=}$};
\node() at (19.3,6) {$\overset{\scriptsize\begin{matrix}\eqref{eq:hex:rm23}\\ \eqref{eq:hex:rm12}\end{matrix}}{=}$};
\node() at (24,-1.5) {\raisebox{5.5pt}{\rotatebox{-90}{=}}\, \scriptsize\eqref{ybe}};
\end{tikzpicture}
\end{center}
\caption{If $X$ generates $G$, and $k|_X$ satisfies \eqref{eq:weird} for all $a,b\in X$, then $k$ satisfies \eqref{eq:weird} for all $a,b\in G$. The proof is by double induction on the number of generators in the expressions of $a$ and $b$. In the picture, we see the case when $a$ is the product of two generators, and $b$ is the product of three generators.}\label{fig:proof_weird_0}
\end{figure}
The following lemma deals with a special case that will be useful in \S\ref{sec:extending}.
\begin{lemma}\label{lem:inverseproperties}
	Let $k$ be a group reflection for $r$. Suppose moreover that $k$ satisfies
	\begin{equation}\label{eq:inversek}\tag{\sc inv} k(a^{-1}) = k(a)^{-1} 
	\end{equation}
	for all $a\in G$. Then $k$ is a left group reflection with respect to $r^{-1}$, and moreover $k(a) =k(b\leftharpoonup a)\rightharpoonup (b\rightharpoonup a) $ for all $b\in G$.
\end{lemma}
\begin{proof}
		Using \eqref{eq:inversek}, one has
		\begin{align*}
			k(ab)&= k(b^{-1}a^{-1})^{-1}\\
			&= \left( (b^{-1}\rightharpoonup k(a)^{-1}) k(b^{-1}\leftharpoonup k(a)^{-1})  \right)^{-1}\\
			&= \left(  (k(a)\leftharpoondown b)^{-1} k(k(a)\rightharpoondown b)^{-1} \right)^{-1}\\
			&= k(k(a)\rightharpoondown b)(k(a)\leftharpoondown b),
		\end{align*}
		which is \eqref{eq:right_reflection-product} with respect to $r^{-1}$; and
		\begin{align*}
			k(b)&= k(b^{-1})^{-1}\\
			&= \left((b^{-1}\rightharpoonup a^{-1})\rightharpoonup k(b^{-1}\leftharpoonup a^{-1})\right)^{-1}\\
			&= \left((a\leftharpoondown b)^{-1}\rightharpoonup k(a\rightharpoondown b)^{-1}\right)^{-1}\\
			&= k(a\rightharpoondown b)\leftharpoondown (a\leftharpoondown b),
		\end{align*}
		which is \eqref{eq:right_weird}. Finally, using the previous properties, we observe that
		\begin{align*}
			k(a)^{-1}&= (a^{-1}\leftharpoondown b^{-1})\leftharpoondown k(a^{-1}\rightharpoondown b^{-1})\\
			&= (b\rightharpoonup a)^{-1}\leftharpoondown k(b\leftharpoonup a)^{-1}\\
			&= \left( k(b\leftharpoonup a)\rightharpoonup (b\rightharpoonup a) \right)^{-1},
		\end{align*}
		whence $k(a) =k(b\leftharpoonup a)\rightharpoonup (b\rightharpoonup a) $.
\end{proof}
\subsection{Group Drinfeld twists} The following is a version of Drinfeld twists for braided groups, defined by Ghobadi \cite{ghobadi2021drinfeld}. The definition is motivated by the description of braided groups as remnants of coquasitriangular Hopf algebras in the category $\mathsf{SupLat}$, see \cite{GHOBADI2021607}.
\begin{definition}[{\cite[Definition 2.2]{ghobadi2021drinfeld}}]
Let $(G,r)$ be a braided group. A (\textit{group}) \textit{Drinfeld twist} for $(G,r)$ is the datum of maps $F\colon G\times G\to G\times G$, $\Phi,\Psi\colon G\times G\times G\to G\times G\times G$, such that $(F,\Phi,\Psi)$ is a set-theoretic Drinfeld twist for $r$, and moreover
\begin{align}
\label{bdt1}\tag{\sc bdt1} &\Phi(1,b,c) = (1,b,c),\quad \Psi(a,b,1) = (a,b,1),\\
\label{bdt2}\tag{\sc bdt2} &F(a,1)=(a,1),\quad F(1,b)=(1,b),\\
\label{bdt3}\tag{\sc bdt3}&m_{23} \Phi = Fm_{23},\\
\label{bdt4} \tag{\sc bdt4}&m_{12} \Psi = Fm_{12}.
\end{align}
\end{definition}
\begin{proposition}[{\cite[Theorem 2.5]{ghobadi2021drinfeld}}] \label{prop:twist_of_braiding_is_braiding}
Let $(G,m,1)$ be a group, with multiplication $m$ and unit $1$. Let $r$ be a braiding on $G$. Given a group Drinfeld twist $(F,\Phi, \Psi)$, one has that $G^F:=(G, mF^{-1}, 1)$ is a group, and the twisted solution $F r F^{-1}$ is a braiding on $G^F$.
\end{proposition}
\begin{proposition}[{\cite[Theorem 2.7 and previous paragraph]{ghobadi2021drinfeld}}] \label{prop:related_iff_same_additive_group}\label{prop:ghobadi_lemma_for_groups}
The precategory having braided groups as vertices, and group Drinfeld twists as arrows, is a groupoid, which we denote by $\mathsf{BGDrin}$. The subgroupoid of braided groups having $G$ as underlying set is denoted by $\mathsf{BGDrin}_G$.

Let $(G,+, \circ)$ and $(G, +', \circ')$ be skew braces: then the corresponding braidings, on $(G,\circ)$ and $(G,\circ')$ respectively, are related by a group Drinfeld twist if and only if $(G,+)\cong (G,+')$.
\end{proposition}


\begin{remark}[{\cite[\S2]{ghobadi2021drinfeld}}]\label{rem:isogroups-are-twist}
Isomorphic braided groups are in particular group Drinfeld twists of each other. Thus the above groupoids can be redefined up to isomorphism class.
\end{remark}
\subsection{Group reflections yield group Drinfeld twists} We now prove a group-theoretic version of Theorem \ref{thm:J-is-Drinfeld}.
\begin{theorem}\label{thm:J-is-braided-Drinfeld}
Let $(G,r)$ be a braided group, and $k$ be a group reflection. Then, the associated guitar map $J = J^{k;2}$ is a group Drinfeld twist for $(G,r)$. As a consequence, $(\refl{G}{k}, \refl{r}{k}):=(G^J, JrJ^{-1}) $ is a braided group.
\end{theorem}
\begin{proof}
It suffices to check the additional conditions \eqref{bdt1}--\eqref{bdt4} for $J$. We get two candidates $\Phi$ and $\Psi$ from the proof of Theorem \ref{thm:J-is-Drinfeld}, which make $J$ into a set-theoretic Drinfeld twist: thus we prove that they also work for the group-theoretic version. We compute
\begin{align*}
\Phi(1,b,c) &= \Big( \big( 1\leftharpoonup (b\rightharpoonup k(c))\big)\leftharpoonup k(b\leftharpoonup k(c)) ,\; b,\; c\Big)\overset{\eqref{eq:1b-b1}}{=} (1,b,c),\end{align*}
and
\begin{align*}
\Psi(a,b,1)&= \Big( a\leftharpoonup(b\rightharpoonup k(1)),\; b\leftharpoonup k(1),\; 1 \Big)\\
\overset{\eqref{eq:reflection-unity}}&{=}\Big( a\leftharpoonup(b\rightharpoonup 1),\; b\leftharpoonup 1,\; 1 \Big)
\overset{\eqref{eq:a1-1a}}{=} (a,b,1),
\end{align*}
which proves \eqref{bdt1}. As for \eqref{bdt2}, one has
\begin{align*}
J(a,1) &= (a\leftharpoonup k(1),1) = (a,1),\\
J(1,b)&= (1\leftharpoonup k(b),b) \stackrel{\eqref{eq:1b-b1}}{=} (1,b).
\end{align*}
We now check \eqref{bdt3}:
\begin{align*}
m_{23}\Phi(a,b,c) &= \Big( \big( a\leftharpoonup (b\rightharpoonup k(c))\big)\leftharpoonup k(b\leftharpoonup k(c)) ,\; bc\Big)\\
\overset{\eqref{eq:hex:rm23}}&{=} \Big(  a\leftharpoonup \big((b\rightharpoonup k(c))\,k(b\leftharpoonup k(c))\big) ,\; bc\Big)\\
\overset{\eqref{eq:reflection-product}}&{=}\Big( \big( a\leftharpoonup k(bc) ,\; bc\Big)\\
&= J m_{23}(a,b,c).\end{align*}We finally check \eqref{bdt4}:\begin{align*}
m_{12}\Psi(a,b,c)&= \Big( \big(a\leftharpoonup(b\rightharpoonup k(c)) \big)(b\leftharpoonup k(c)),\; c \Big)\\
\overset{\eqref{eq:hex:rm12}}&{=}\big(ab\leftharpoonup k(c),\; c \big)\\
&= Jm_{12}(a,b,c).\qedhere
\end{align*}
\end{proof}
Once again, the twisted structure $(\refl{G}{k}, \refl{r}{k})$ depends exclusively on the maps $\rho^{}_{k(a)}$; thus \textit{a priori} \eqref{eq:reflection-unity}--\eqref{eq:weird} are sufficient but not necessary in orded for $(\refl{G}{k}, \refl{r}{k})$ to be a braided group. However, we now prove that \eqref{eq:reflection-unity} and \eqref{eq:reflection-product} are `optimal' conditions, in the sense of the following theorem.
\begin{theorem}\label{thm:br-optimal}
Let $(G,r)$ be a braided group, with $\leftharpoonup$ a faithful action, and $k$ be a map $G\to G$. Then the $k$-twisted structure $\refl{G}{k}$ is a group if and only if $k$ satisfies \eqref{eq:reflection-unity} and \eqref{eq:reflection-product}; and $(\refl{G}{k}, \refl{r}{k})$ is a braided group if and only if $k$ also satisfies 
\begin{equation}\label{eq:involution}\tag{\textsc{bre}3$'$} k(a\,k(b))^2 = 1 \end{equation}
 for all $a,b\in G$. In particular, \eqref{eq:weird} implies \eqref{eq:involution}.
\end{theorem}
\begin{proof}
We first impose that $\refl{G}{k}$ is a group. The `if' part would already follow from Theorem \ref{thm:J-is-braided-Drinfeld}, but we need to perform the explicit computations anyway, in order to prove the `only if'. The associativity of $\refl{\cdot}{k}$ translates as 
\begin{align*}
(a\refl{\cdot}{k} b)\refl{\cdot}{k}c&= \Big(a\leftharpoonup k(b)^{-1}(b\rightharpoonup k(c)^{-1})\Big)(b\leftharpoonup k(c)^{-1})c\\
\overset{!}&{=} \Big( a\leftharpoonup k((b\leftharpoonup k(c)^{-1})c)^{-1} \Big) (b\leftharpoonup k(c)^{-1})c = a\refl{\cdot}{k} (b\refl{\cdot}{k}c),
\end{align*}
where $\overset{!}{=}$ denotes the equalities that we impose. By cancelling $(b\leftharpoonup k(c)^{-1})c$ on the right, and by applying the faithfulness of $\leftharpoonup$, we get \[k(b)^{-1}(b\rightharpoonup k(c)^{-1})= k((b\leftharpoonup k(c)^{-1})c)^{-1}.\] Through the bijective substitution $b:= \beta\leftharpoonup k(c)$, this becomes
\begin{align*} k(\beta c) \overset{!}&{=} \big( (\beta\leftharpoonup k(c))\rightharpoonup k(c)^{-1}\big)^{-1}k(\beta\leftharpoonup k(c))\\
\overset{\eqref{eq:hex:rm23}}&{=} (\beta\rightharpoonup k(c))\;k(\beta\leftharpoonup k(c)),  \end{align*}
which is \eqref{eq:reflection-product}. Thus \eqref{eq:reflection-product} is equivalent to the associativity of $\refl{\cdot}{k}$. We then impose the existence of a neutral element: since $1\refl{\cdot}{k} a = a$ always holds, this neutral element must be $1$. We thereby impose \[a\,\refl{\cdot}{k}1 = (a\leftharpoonup k(1)^{-1})1 \overset{!}{=} a,\]
and by faithfulness of $\leftharpoonup$ this is equivalent to $k(1)= 1$. Thus \eqref{eq:reflection-unity} is equivalent to the existence of a neutral element. It is immediate to verify that $a$ has the left inverse $a^{-1}\leftharpoonup k(a)$, which is therefore a bilateral inverse.

We finally impose that $\refl{r}{k}$ is a braiding on $\refl{G}{k}$. First, notice that \begin{equation}\label{eq:twisted-braided-comm}(a\,\refl{\rightharpoonup}{k} b)\,\refl{\cdot}{k}(a\,\refl{\leftharpoonup}{k}b) = a\,\refl{\cdot}{k} b\end{equation} always holds. Thus, it is immediate to see that $\refl{r}{k}$ is a braiding if and only if the following two conditions hold:
\begin{align}
\label{rkbraiding1}& a\,\refl{\leftharpoonup}{k}(b\,\refl{\cdot}{k} c) = (a\,\refl{\leftharpoonup}{k} b)\,\refl{\leftharpoonup}{k} c,\\ &
\label{rkbraiding2} (a\,\refl{\cdot}{k} b)\,\refl{\leftharpoonup}{k} c = (a\,\refl{\leftharpoonup}{k} (b\,\refl{\rightharpoonup}{k}c))\,\refl{\cdot}{k} (b\,\refl{\leftharpoonup}{k}c).
\end{align}Indeed, the other two conditions
\begin{align*}
& (a\,\refl{\cdot}{k} b) \,\refl{\rightharpoonup}{k}c= a\,\refl{\rightharpoonup}{k} (b\,\refl{\rightharpoonup}{k} c),\\ &
a\,\refl{\rightharpoonup}{k} (b\,\refl{\cdot}{k} c) = (a\,\refl{\rightharpoonup}{k} b)\,\refl{\cdot}{k} ((a\, \refl{\leftharpoonup}{k} b)\,\refl{\rightharpoonup}{k}c)
\end{align*}
just follow from \eqref{eq:twisted-braided-comm}--\eqref{rkbraiding2}.

We unwrap the condition \eqref{rkbraiding1}:
\begin{align*}
a\refl{\leftharpoonup}{k}(b\,\refl{\cdot}{k} c)&=  a\leftharpoonup\Big( k\big( (b\leftharpoonup k(c)^{-1}) c  \big) (b\leftharpoonup k(c)^{-1})c\Big) \\
\overset{!}&{=} a\leftharpoonup k(b)^{-1}b k(c)^{-1} c = (a\refl{\leftharpoonup}{k} b)\refl{\leftharpoonup}{k} c.
\end{align*}
Thus, by faithfulness of $\rho$, we obtain that $\refl{\leftharpoonup}{k}$ is a right action if and only if \[   k\big( (b\leftharpoonup k(c)^{-1}) c  \big) (b\leftharpoonup k(c)^{-1})c = k(b)^{-1}b k(c)^{-1} c\]
holds for all $b,c\in G$. Through the bijective substitution $\beta = b\leftharpoonup k(c)^{-1}$ we obtain
\[ k(\beta c)\beta c = k(\beta\leftharpoonup k(c))^{-1} (\beta\leftharpoonup k(c)) k(c)^{-1}c ,\]
which, by cancelling $c$ on the right, and rearranging under \eqref{eq:reflection-product}, is equivalent to 
\[ k(\beta k(c))^2 =1\text{ for all }\beta, c\in G,\]
which is \eqref{eq:involution}. When $k$ satisfies \eqref{eq:reflection-unity}--\eqref{eq:weird}, in particular $(\refl{G}{k}, \refl{r}{k})$ is a braided group by Theorem \ref{thm:J-is-braided-Drinfeld}, thus (under the assumption that $\leftharpoonup$ is faithful) \eqref{eq:involution} must hold. When $\leftharpoonup$ is not faithful, the same proof shows that $\rho^{}_{k(a\,k(b))^2} = \rho^{}_1$ must hold.

We finally expand condition \eqref{rkbraiding2}, and observe that it is automatically satisfied under the assumption of \eqref{eq:reflection-product}. The left-hand side of \eqref{rkbraiding2} is
\[ (a\,\refl{\cdot}{k} b)\,\refl{\leftharpoonup}{k} c = \Big( (a\leftharpoonup k(b)^{-1})\leftharpoonup (b\rightharpoonup k(c)^{-1}c) \Big)(b\leftharpoonup k(c)^{-1} c)  . \]
The right-hand side of \eqref{rkbraiding2} is
\begin{align*}
& (a\,\refl{\leftharpoonup}{k} (b\,\refl{\rightharpoonup}{k}c))\,\refl{\cdot}{k} (b\,\refl{\leftharpoonup}{k}c)\\
&= \Bigg( a\leftharpoonup \bigg( k\Big(((b\leftharpoonup k(c)^{-1})\rightharpoonup c) \leftharpoonup k(b\leftharpoonup k(c)^{-1}c)\Big)^{-1}\\
&\hspace{1em}\Big(((b\leftharpoonup k(c)^{-1})\rightharpoonup c) \leftharpoonup k(b\leftharpoonup k(c)^{-1}c)\Big) k(b\leftharpoonup k(c)^{-1}c)^{-1} \bigg) \Bigg) (b\leftharpoonup k(c)^{-1}c).
\end{align*}
By cancelling $(b\leftharpoonup k(c)^{-1}c)$ on the right, and by applying the bijective substitution $\beta:=b\leftharpoonup k(c)^{-1}$, the condition \eqref{rkbraiding2} becomes
\begin{align*}
& a\leftharpoonup \Big( k(\beta\leftharpoonup k(c))^{-1} ((\beta\leftharpoonup k(c))\rightharpoonup k(c)^{-1}c) \Big) \\
&=a\leftharpoonup \Big( k\big( (\beta\rightharpoonup c)\leftharpoonup k(\beta\leftharpoonup c) \big)^{-1} \big( (\beta\rightharpoonup c)\leftharpoonup k(\beta\leftharpoonup c) \big) k(\beta\leftharpoonup c)^{-1} \Big)
\end{align*}
which, since $\leftharpoonup$ is faithful, is equivalent to 
\begin{align*}
& k(\beta\leftharpoonup k(c))^{-1} ((\beta\leftharpoonup k(c))\rightharpoonup k(c)^{-1}c)  \\
&= k\big( (\beta\rightharpoonup c)\leftharpoonup k(\beta\leftharpoonup c) \big)^{-1} \big( (\beta\rightharpoonup c)\leftharpoonup k(\beta\leftharpoonup c) \big) k(\beta\leftharpoonup c)^{-1}\\
\iff & k(\beta \leftharpoonup k(c))^{-1} ((\beta\leftharpoonup k(c))\rightharpoonup k(c)^{-1})(\beta\rightharpoonup c)
= k\big( (\beta\rightharpoonup c)\leftharpoonup k(\beta\leftharpoonup c) \big)^{-1}\\ &\big( (\beta\rightharpoonup c)\rightharpoonup k(\beta\leftharpoonup c) \big)^{-1}(\beta\rightharpoonup c)k(\beta\leftharpoonup c)k(\beta\leftharpoonup c)^{-1}\\
\iff &k(\beta \leftharpoonup k(c))^{-1} ((\beta\leftharpoonup k(c))\rightharpoonup k(c)^{-1})\\
&{=} k\big( (\beta\rightharpoonup c)\leftharpoonup k(\beta\leftharpoonup c) \big)^{-1}\big( (\beta\rightharpoonup c)\rightharpoonup k(\beta\leftharpoonup c) \big)^{-1}\\
\overset{\eqref{eq:reflection-product}}{\iff}& 
((\beta\leftharpoonup k(c))\rightharpoonup k(c)^{-1}) {=} k(\beta\leftharpoonup k(c)) k((\beta\rightharpoonup c)(\beta\leftharpoonup c))^{-1}\\
\overset{\eqref{eq:reflection-product}}{\iff} & ((\beta\leftharpoonup k(c))\rightharpoonup k(c)^{-1}) = (\beta \rightharpoonup k(c))^{-1} k(\beta c)k(\beta c)^{-1}\\
\iff & (\beta\rightharpoonup k(c)) ((\beta\leftharpoonup k(c))\rightharpoonup k(c)^{-1})= 1\\
\iff & \beta \rightharpoonup k(c)k(c)^{-1} = 1,
\end{align*}
and the latter is always true.
\end{proof}
\begin{corollary}
Let $(G,r)$ be a braided group, and $k$ be a bijective group reflection for $r$. Then $k(a)^2$ acts trivially on the right, for all $a\in G$. If moreover $\leftharpoonup$ is faithful, then $G$ is $2$-torsion.
\end{corollary}
\begin{proof}
Every $a\in G$ can be written as $b\, k(c)$ for some $b, c\in G$. Thus $\rho^{}_{k(b\,k(c))^2} = \rho^{}_1$ implies $\rho^{}_{k(a)^2} = \rho^{}_1$ for all $a\in G$, and we conclude by bijectivity of $k$.
\end{proof}

In the case of braided groups, we can get some further insight on Question \ref{quest:composing_guitar_maps}.
\begin{remark}\label{rem:goodcandidate}
	Let $(G,r)$ be a braided group, $k$ be a group reflection for $(G,r)$, and $h$ be a group reflection for $(\refl{G}{k}, \refl{r}{k})$. Define $\ell(a):= k(a) hk(a)^{-1} h(a)$. By Corollary \ref{cor:composition}, $\ell$ is a good candidate to be the `composition' of $k$ and $h$, i.e., to have $\refl{r}{\ell} = \refl{(\refl{r}{k})}{h}$---here we use the notation $\refl{r}{\ell}$ even when $\ell$ is not necessarily an $r$-reflection. Indeed, $\refl{r}{\ell} = \refl{(\refl{r}{k})}{h}$ holds by construction, but it is generally false that $\ell$ is a reflection. 
\end{remark}
In most of the examples of small size, $\ell(a):= k(a) hk(a)^{-1} h(a)$ defines in fact a group reflection for $(G,r)$, and also a group reflection for $(\refl{G}{k}, \refl{r}{k})$. But this is not always true: the first counterexamples to both statements appear in order $6$; and the first  counterexamples with $k$ bijective appear in order $8$. These have been computed with \textsc{gap} \cite{gap4} and the \texttt{YangBaxter} package \cite{YangBaxterGAP}.
\begin{example}\label{ex:ell_not_refl}
	We use the notation $G_{n,j}$ for the skew brace located in position $(n,j)$ in the database \texttt{SmallSkewbrace} \cite{YangBaxterGAP}. Every skew brace $G_{n,j}$ is seen as a set with elements $1,\dots, n$, through the \textsc{gap} function \texttt{AsList}; thus when we write $k = (a_1, \dots, a_n)$, we mean that $k$ is the map $G_{n,j}\to G_{n,j}$ that acts like $i\mapsto a_i$ once $G_{n,j}$ has been identified with $\set{1,\dots, n}$ using \texttt{AsList}. We use the notation $k_i$ for the reflections, $h_i$ for the reflections on the skew braces twisted by $k_i$, and $\ell_i$ for the map obtained from $k_i$ and $h_i$ as in Remark \ref{rem:goodcandidate}.
	
	For $G_{6,1}$, the map $k_1= (
	1,1,1,4,4,4)$ is a group reflection, and $h_1=(
	1,2,3,1,2,3
	)$ is a group reflection for $G^{(k_1)}_{6,1}$, but $\ell_1 =\id $ is not a reflection for $G_{6,1}$ (indeed, the solution on $G_{6,1}$ is not involutive), but it is a reflection for $G_{6,1}^{(k)}$. 
	
	For $G_{6,5}$, the map $k_2=\id$ is a reflection, and the map $h_2=(1,1,1,1,1,1)$ is a reflection on $G_{6,5}^{(k_2)}$, and $\ell_2 = \id$ is a reflection for $G_{6,5}$ but not for $G_{6,5}^{(k_2)}$.
	
	For $G_{8,34}$, the bijective map $k_3=(
	1,2,5,6,7,8,3,4)$ is a reflection, and $h_3=\id $ is a reflection for $G_{8,34}^{(k_3)}$, but $\ell_3 = \id$ is not a reflection for $G_{8,34}$.
\end{example}

\subsection{On the groupoid of group Drinfeld twists}\label{sec:groupoid} Every group Drinfeld twist can be obtained as the composition of two twists of two special kinds. 
\begin{proposition}[{\cite[dual of Corollary 2.8]{ghobadi2021drinfeld}}]\label{prop:decomposition_arrows}
	Any arrow $(G,\cdot, r)\to (G, \cdot', r')$ in the groupoid $\mathsf{BGDrin}_G$ can be decomposed as $g\circ f$, where
	\begin{enumerate}[label={\upshape(\Roman*)},ref={\upshape(\Roman*)}]
		\item\label{item:type1} the arrow $f$ is a twist between trivial skew braces: namely, between two braidings of the form $(a,b)\mapsto (b, b^{-1}ab)$ for two different group structures;
		\item\label{item:type2} the arrow $g$ is a twist of the form $(a,b)\mapsto (\varrho_b(a), b)$ for some map $\varrho$.
	\end{enumerate}
\end{proposition}
As a consequence of Proposition \ref{prop:related_iff_same_additive_group}, if two trivial skew braces $(G,\cdot,\cdot)$ and $ (G,\bar{\cdot},\bar{\cdot})$ are related by a twist of type \ref{item:type1}, then $(G,\cdot)$ and $(G,\bar{\cdot})$ are isomorphic. 

Using this fact, the arrows of type \ref{item:type1} in Proposition \ref{prop:decomposition_arrows} have been completely classified.
\begin{proposition}[{\cite[\S2]{ghobadi2021drinfeld}}]\label{prop:TwistsOfType1}
	There is a bijective correspondence between group Drinfeld twists $F\colon (G,\cdot,\cdot)\to (G,\bar{\cdot},\bar{\cdot})$ of type \ref{item:type1}, and families $\set{ f_x }_{x\in G}$ of isomorphisms $f_x\colon (G,\cdot)\to (G,\bar{\cdot})$ satisfying $f_x(x) = x$ for all $x\in G$. The correspondence is given by
	\[ F(x,y) = (f_{xy}(x), f_{xy}(y)). \]
\end{proposition}
In the above Proposition, we are calling $F$ alone `a Drinfeld twist' without specifying the maps $\Phi$ and $\Psi$. This imprecision is actually safe, since $F$ turns out to be a group Drinfeld twist for unique maps $\Phi$ and $\Psi$, determined in the proof of \cite[Theorem 2.7]{ghobadi2021drinfeld}.
\begin{remark}\label{rem:trivialbrace}
	Let $(G,\cdot,\cdot)$ be a trivial skew brace, with its braiding. As an immediate consequence of Proposition \ref{prop:TwistsOfType1}, every reflection $k$ corresponds to the Drinfeld twist $F = \id$. 
	
	We can easily classify the group reflections $k$ on $(G,\cdot,\cdot)$: here \eqref{eq:weird} becomes $k(a) = k(b^{-1}ab)$, and \eqref{eq:reflection-product} becomes \[k(ab) = k(b)k(k(b)^{-1}ak(b)) = k(b)k(a).\] Therefore, the group reflections are the class-functions anti-ho\-mo\-mor\-phisms. In particular, from \eqref{eq:weird} one has $k(b)k(a) k(b)^{-1}= k(a)$, thus the image of $k$ is an abelian subgroup of $G$, and hence $k$ is also a homomorphism.
\end{remark}

\begin{remark}\label{rem:rack-type-braided-groups}The derived solution $r'$ has a special form, which can be extrapolated into a general definition. A set with a binary operation $\triangleleft$ is called a \emph{shelf} if the self-distributivity rule
	\[ (a\triangleleft b)\triangleleft c = (a\triangleleft c)\triangleleft (b\triangleleft c) \]
	is satisfied. A shelf $(X,\triangleleft)$ is called a \emph{rack} if $\blank\triangleleft a$ is a bijection for all $a$. It is known that each shelf produces a solution $r\colon (a,b) \mapsto (b, a\triangleleft b)$, which is non-degenerate if and only if $(X,\triangleleft)$ is a rack. Reflections on such solutions have been investigated in detail by Albano, Mazzotta, and Stefanelli \cite{albano2024reflections}.
	
	A braiding $r$ on a group $G$ is a rack solution if and only if $a\triangleleft b = b^{-1}ab$: this follows immediately from \eqref{eq:mult}. As a consequence:
	\begin{enumerate}
		\item Proposition \ref{prop:decomposition_arrows} classifies the group Drinfeld twists between a braided group of rack type, and any braided group;
		\item a group Drinfeld twist between two braided groups \textit{both} of rack type must be an arrow of type \ref{item:type1}.
	\end{enumerate}
\end{remark}

%

 The composition of the arrows of type \ref{item:type1} has a neat algebraic interpretation. 
\begin{lemma}
	The composition of the twist $F_1$ given by $\set{f_x\colon (G,\cdot)\to (G,\bar{\cdot})}_{x\in G}$ with the twist $F_2$ given by $\set{g_x\colon (G,\bar{\cdot})\to (G, \bar{\bar{\cdot}})}_{x\in G}$, is the twist $ F_2\circ F_1$ given by $\set{g_x\circ f_x}_{x\in G}$. 
\end{lemma}
\begin{proof}
	The composition of homomorphisms is a homomorphism, and moreover one has $g_x\circ f_x(x) = g_x(x) = x$, thus $\set{g_x\circ f_x}_{x\in G}$ defines indeed a Drinfeld twist. We conclude by observing that
	\begin{align*}
		F_2\circ F_1(x,y) &= \left( g_{f_{xy}(x)\bar{\cdot}f_{xy}(y)}(f_{xy}(x)) , \; g_{f_{xy}(x)\bar{\cdot}f_{xy}(y)}(f_{xy}(y))  \right)\\
		&= \left( g_{f_{xy}(xy)}f_{xy}(x) , \; g_{f_{xy}(xy)}f_{xy}(y)  \right)\\
		&= \left( g_{xy}f_{xy}(x), g_{xy}f_{xy}(y)\right).\qedhere
	\end{align*}
\end{proof}


We now classify the twists of type \ref{item:type2}. These are what we may call \textit{one-legged Drinfeld twists}, i.e.\@ twists $F$ that only act on one component. In particular, all the reflection-twists have this form.
\begin{proposition}\label{prop:one_legged_twist}
	Let $(G,r)$ be a braided group, and $J\colon (a,b)\mapsto (\varrho_b(a), b)$ be an endomap of $G\times G$ for some map $\varrho\colon G\times G\to G$, $(a,b)\mapsto \varrho_b(a)$, and satisfying moreover $\varrho_1 =\id[G]$ and $\varrho_a(1) = 1$ for all $a\in G$. Then $J$ is a group Drinfeld twist for $r$ if and only if $\varrho$ satisfies
	\begin{align}
		&\label{eq:rho_bijective}\varrho_b \colon G\to G \text{ is bijective for all } b,\\
		&\label{eq:rho_unitary}\varrho_1 =\id[G]\text{ and }\varrho_a(1) = 1 \text{ for all } a\in G,\\
		&\label{eq:dt1.1-rho} \varrho_{\varrho_c(b)}\Big( \varrho_c(ab)\varrho_c(b)^{-1} \Big) = \varrho_{bc}(a),\\
		&\label{eq:dt2.1-rho} \varrho_c(ab) \varrho_c(a\leftharpoonup b)^{-1} = \Big( \varrho_c(ab)\varrho_c(b)^{-1} \Big)\rightharpoonup \varrho_c(b), \\
		&\label{eq:dt2.2-rho} \varrho_c(a\leftharpoonup b)=\Big( \varrho_c(ab)\varrho_c(b)^{-1}\Big)\leftharpoonup \varrho_c(b),
	\end{align}
	and in this case one necessarily has
	\begin{align}
		\label{eq:phi-rho}&\Phi(a,b,c) =  \Big( \varrho_{bc}(a),\; b,\; c  \Big),\\
		\label{eq:psi-rho}&\Psi(a,b,c) = \Big( \varrho_c(ab)\cdot\varrho_c(b)^{-1},\; \varrho_c(b),\; c \Big).
	\end{align}
\end{proposition}
\begin{proof}
	Suppose that $J$ is a group Drinfeld twist. Observe that \eqref{eq:rho_bijective} must hold, in order for $J$ to be bijective. Equation \eqref{bdt2} yields immediately \eqref{eq:rho_unitary}. We search for suitable maps $\Phi$ and $\Psi$, and thus write
	\[ \Phi(a,b,c) = \Big( \Phi_1(a,b,c), \Phi_2(a,b,c), \Phi_3(a,b,c) \Big), \]
	and similarly for $\Psi$. Then, \eqref{bdt3} and \eqref{bdt4} yield
	\begin{align*} &\Phi_1(a,b,c) = \varrho_{bc}(a), && \Phi_2(a,b,c)\Phi_3(a,b,c) = bc,\\ & \Psi_1(a,b,c)\Psi_2(a,b,c) = \varrho_c(ab),&& \Psi_3(a,b,c) = c, \end{align*}
	while \eqref{bdt1} yields
	\begin{align*}
		& \Phi_1(1,b,c) = 1,&& \Phi_2(1,b,c) = b, &&\Phi_3(1,b,c) = c,
		\\
		& \Psi_1(a,b,1) = a,&& \Psi_2(a,b,1) = b, &&\Psi_3(a,b,1) = 1.
	\end{align*}
	From \eqref{dt1} we obtain
	\begin{align*}&\varrho_{\Psi_2(a,b,c)}\Psi_1(a,b,c)= \Phi_1(a,b,c),\\ & \Psi_2(a,b,c) = \varrho_{\Phi_3(a,b,c)}\Phi_2(a,b,c),\\ & \Psi_3(a,b,c) = \Phi_3(a,b,c), \end{align*}
	whence also
	\begin{align*}
		&\Phi_3(a,b,c) = c,\\  &\varrho_{bc}(a) = \varrho_{\varrho_c\Phi_2(a,b,c)}\Psi_1(a,b,c).\end{align*}
	Now, $ \Phi_3(a,b,c)= c$ together with $\Phi_2(a,b,c)\Phi_3(a,b,c) = bc$ imply\[ \Phi_2(a,b,c) = b,\]
	and hence
	\[ \Psi_1(a,b,c)\cdot \varrho_c(b) = \varrho_c(ab)\;\implies \;\Psi_1(a,b,c) = \varrho_c(ab) \cdot \varrho_c(b)^{-1}. \]
	In conclusion, we have obtained \eqref{eq:phi-rho} and \eqref{eq:psi-rho}.
	
	We can substitute these expressions in the axioms of a group Drinfeld twist, thus getting a set of conditions on $\varrho$. By substituting in \eqref{dt1} we get three conditions, of which the only nontrivial one is \eqref{eq:dt1.1-rho}. By doing the same for \eqref{dt2}, the only nontrivial conditions that we get are \eqref{eq:dt2.1-rho} and \eqref{eq:dt2.2-rho}; while substituting \eqref{dt3} we only get trivial conditions. The axioms  \eqref{bdt1}--\eqref{bdt4} are automatically satisfied assuming \eqref{eq:rho_bijective}--\eqref{eq:dt2.2-rho}.
	
	Conversely, it is clear that $J$ is a Drinfeld twist, with $\Phi$ and $\Psi$ defined as in \eqref{eq:phi-rho} and \eqref{eq:psi-rho} respectively, if and only if \eqref{eq:rho_bijective}--\eqref{eq:dt2.2-rho} hold.
\end{proof}
\begin{corollary}
	If $J\colon (a,b)\mapsto (a\leftharpoonup k(b), b)$ is a guitar map for some group reflection $k$, then $J$ is a group Drinfeld twist for \emph{unique} maps $\Phi$ and $\Psi$: namely, the maps defined in \eqref{eq:phi-rho} and \eqref{eq:psi-rho} with $\varrho_b(a):= a\leftharpoonup k(b)$.
\end{corollary}

Proposition \ref{prop:one_legged_twist} takes an easier form in the special case of an abelian group $G$, endowed with the canonical flip $\tau\colon (a,b)\mapsto (b,a)$. It is known that $(G,\tau)$ is a braided group---and, in fact, $\tau$ is a braiding on a group $G$ if and only if $G$ is abelian.

\begin{proposition}
	Let $G$ be an abelian group, with braiding given by the canonical flip $\tau$. Then $(a,b)\mapsto (\varrho_b(a), b)$ is a group Drinfeld twist of type \ref{item:type2} if and only if $\varrho_c$ is a group isomorphism for all $c\in G$, and moreover \begin{equation}\label{eq:dt1_flip_case}\varrho_{bc}(a) = \varrho_{\varrho_c(b)}(\varrho_c(a))\end{equation} holds for all $a,b,c\in G$.
\end{proposition}
\begin{proof}
	Condition \eqref{eq:dt2.2-rho} implies that $\varrho_c$ is a homomorphism, which is bijective by \eqref{eq:rho_bijective}; while \eqref{eq:dt1_flip_case} follows from \eqref{eq:dt1.1-rho}. Conversely, assuming that $\set{\varrho_c}_{c\in G}$ is a family of group isomorphisms satisfying \eqref{eq:dt1_flip_case}, all the other conditions of Proposition \ref{prop:one_legged_twist} follow immediately.
\end{proof}

\section{Extending reflections on the structure group of a solution}\label{sec:extending} 
It is well known that, given a Yang--Baxter map $r$ on $X$, we can construct a map $\tilde{r}$ on the free monoid $X^*$ on $X$, as in Figure \ref{fig:extension-r}. This map $\tilde{r}$ descends to a well defined map $\bar{r}$ on the structure monoid $M(X,r)$. 

If $r$ is bijective non-degenerate, then $\tilde{r}$ can be extended from the free monoid $X^*$ to the free group $\Free(X)$ on $X$, and this extension descends, in turn, to a map (which we call again $\bar{r}$) on the structure group $G(X,r)$. There is a unique way to extend $\tilde{r}$ to $\Free(X)$ so that it descends to a braiding on $G(X,r)$ \cite[Theorem 4]{LYZ}: if $\tilde{r}(a,b)$ is defined as $(a\,\tilde{\rightharpoonup}\, b, a\,\tilde{\leftharpoonup}\, b)$, then one defines
\begin{align*} &a^{-1}\tilde{\rightharpoonup}\, b:=( a\rightharpoonup \blank)^{-1}(b), && a\,\tilde{\leftharpoonup}\, b^{-1}:= (\blank\leftharpoonup b)^{-1}(a),\\
&a\,\tilde{\rightharpoonup} \,b^{-1} := ((a\,\tilde{\leftharpoonup}\,b^{-1})\rightharpoonup b)^{-1}, && a^{-1}\tilde{\leftharpoonup}\, b := ( a\leftharpoonup (a^{-1}\tilde{\rightharpoonup}\,b) )^{-1},
 \end{align*}
 \[ \tilde{r}(a^{-1}, b^{-1}) =  \big( (b\,\leftharpoondown a)^{-1}, (b \rightharpoondown a)^{-1}  \big), \]
 where $a,b$ lie in $X$, and $(a,b)\mapsto (a\rightharpoondown b, a\leftharpoondown b)$ is the map $r^{-1}\colon X\times X\to X\times X$. 

Similarly, we may extend a reflection $k$ to a map $\tilde{k}$ on the free monoid on $X$, as in Figure \ref{fig:extension-k}. In \cite{lebed2022reflection}, our $\tilde{k}$ is called the \textit{$k$-Garside map}: in the notation of \cite{lebed2022reflection}, we have $\tilde{k}|_{X^n} = \Delta^{n;k}$. By the relation
$\Delta^{n;k}r_i = r_{n-i}\Delta^{n;k}$
\cite[Theorem 1.8]{lebed2022reflection}, this $\tilde{k}$ also descends to the structure monoid $M(X,r)$, and induces a map $\bar{k}\colon M(X,r)\to M(X,r)$.
\begin{figure}[t]
\begin{subfigure}{0.4\linewidth}
\begin{center}
\begin{tikzpicture}[x=0.5cm, y=0.5cm]
\node (a1) at (-4,3) {$a_1$};
\node (a2) at (-3,3) {$a_2$};
\node (a3) at (-2,3) {$a_3$};
\node (a4) at (-1,3) {$a_4$};
\node (b1) at (2,3) {$b_1$};
\node (b2) at (3,3) {$b_2$};
\node (b3) at (4,3) {$b_3$};
\node (B1) at (-4,-3) {};
\node (B2) at (-3,-3) {};
\node (B3) at (-2,-3) {};
\node (A1) at (1,-3) {};
\node (A2) at (2,-3) {};
\node (A3) at (3,-3) {};
\node (A4) at (4,-3) {};
\draw[draw=white, double=black, very thick] (b1) to (B1);
\draw[draw=white, double=black, very thick] (b2) to (B2);
\draw[draw=white, double=black, very thick] (b3) to (B3);
\draw[draw=white, double=black, very thick] (a1) to (A1);
\draw[draw=white, double=black, very thick] (a2) to (A2);
\draw[draw=white, double=black, very thick] (a3) to (A3);
\draw[draw=white, double=black, very thick] (a4) to (A4);
\end{tikzpicture}
\end{center}
\caption{Extension of $r$ to $\Free(X)$.}\label{fig:extension-r}\end{subfigure}\begin{subfigure}{0.4\linewidth}
\begin{center}
\begin{tikzpicture}[x=0.6cm, y=0.6cm]
\node (a1) at (-2,3) {$a_1$};
\node (a2) at (-1,3) {$a_2$};
\node (a3) at (0,3) {$a_3$};
\node (a4) at (1,3) {$a_4$};
\node (A1) at (1,-2) {};
\node (A2) at (0,-2) {};
\node (A3) at (-1,-2) {};
\node (A4) at (-2,-2) {};
\draw[draw=white, double=black, very thick] (a4)--(2,2)--(A4);
\draw[draw=white, double=black, very thick] (a3)--(2,1)--(A3);
\draw[draw=white, double=black, very thick] (a2)--(2,0)--(A2);
\draw[draw=white, double=black, very thick] (a1)--(2,-1)--(A1);
\draw[draw=gray, line width=3pt] (2.1,3) to (2.1,-2);
\end{tikzpicture}
\end{center}
\caption{Extension of $k$ to $\Free(X)$.}\label{fig:extension-k}
\end{subfigure}\caption{Graphic depiction of how $r$ and $k$ are extended to the free group $\Free(X)$ and to the structure group $G(X,r)$.}\label{fig:extensions}\end{figure}
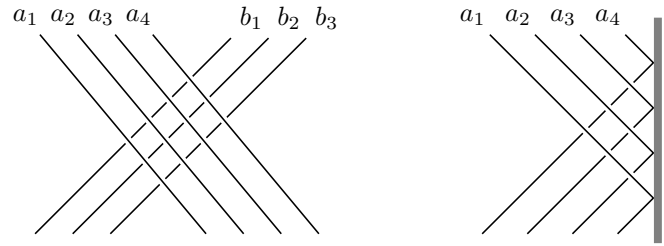
\begin{lemma}\label{lem:extend_reflection}
The map $\bar{k}$ is a set-theoretic reflection for the map $\bar{r}$ on the structure monoid $M(X,r)$.
\end{lemma}
\begin{proof}
The \textsc{re} for $\bar{k}$ is proven inductively. A graphical proof is given in Figure \ref{fig:proof}: consider $k_2 rk_2 r(a_1a_2\ldots a_n,\, b_{n+1}b_{n+2}\ldots b_m)$, move the first $n$ strands out of the way, and apply the \textsc{re} in sequence to bring the bouncing of the $(n+1)$st strand to the top, past the bouncings of the first $n$ strands; then move the first $n$ strands back in place. Repeat for the $(n+2)$nd strand, for the $(n+3)$rd, and so on.
\end{proof}

\begin{figure}[t]
\begin{center}
\begin{tikzpicture}[x=0.4cm, y=0.4cm]
\draw[draw=white, double=black, line width=2pt] (4,11) to [out=-135, in=90] (0,6);
\draw[draw=white, double=black, line width=2pt] (2,11)--(6,8) -- (0,0);
\draw[draw=white, double=black, line width=2pt] (0,11) --(6,5.6) -- (2,0);
\draw[draw=white, double=black, line width=2pt] (0,6)to[out=-90, in=140](6,2)--(4,0);
\draw[draw=gray, line width=3pt] (6.1,0) to (6.1,11);
\draw[draw=white, double=black, line width=2pt] (14,11) --(14,9);
\draw[draw=white, double=black, line width=2pt] (12,11)--(16,8) to[out=150, in=90] (10,0);
\draw[draw=white, double=black, line width=2pt] (14,9)--(14,5.6);
\draw[draw=white, double=black, line width=2pt] (10,11) --(16,5.6)--(12,0);
\draw[draw=white, double=black, line width=2pt] (14,5.6) to[out=-90, in=140] (16,2)--(14,0);
\draw[draw=gray, line width=3pt] (16.1,0) to (16.1,11);
\draw[draw=white, double=black, line width=2pt] (24,11) --(24,9);
\draw[draw=white, double=black, line width=2pt] (22,11)--(26,8) to[out=150, in=90] (20,0);
\draw[draw=white, double=black, line width=2pt] (24,9) to[out=-90, in=134](26,5);
\draw[draw=white, double=black, line width=2pt] (26,5) to[out=-130, in=90] (24,3);
\draw[draw=white, double=black, line width=2pt] (20,11) --(26,3)--(22,0);
\draw[draw=white, double=black, line width=2pt] (24,3) to[out=-90, in=90] (24,0);
\draw[draw=gray, line width=3pt] (26.1,0) to (26.1,11);
\draw[draw=white, double=black, line width=2pt](26,-8)--(20,-14);
\draw[draw=white, double=black, line width=2pt] (24,-3) --(26,-4)to[out=-140,in=90] (24,-6);
\draw[draw=white, double=black, line width=2pt] (22,-3)--(26,-8);
\draw[draw=white, double=black, line width=2pt] (24,-6) to[out=-90, in=90] (25.5,-9) to[out=-90, in=90] (24,-11);
\draw[draw=white, double=black, line width=2pt] (20,-3) --(26,-11)--(22,-14);
\draw[draw=white, double=black, line width=2pt] (24,-11) to[out=-90, in=90] (24,-14);
\draw[draw=gray, line width=3pt] (26.1,-14) to (26.1,-3);
\draw[draw=white, double=black, line width=2pt](16,-8)--(10,-14);
\draw[draw=white, double=black, line width=2pt] (14,-3) --(16,-4)to[out=-140,in=90] (11,-9);
\draw[draw=white, double=black, line width=2pt] (12,-3)--(16,-8);
\draw[draw=white, double=black, line width=2pt] (16,-11)--(12,-14);
\draw[draw=white, double=black, line width=2pt] (11,-9) to[out=-90, in=120] (14,-14);
\draw[draw=white, double=black, line width=2pt] (10,-3) --(16,-11);
\draw[draw=gray, line width=3pt] (16.1,-14) to (16.1,-3);
\node() at (8,5.5){$\stackrel{\text{\textsc{ybe}}}{\rightarrow}$};
\node() at (18,5.5){$\stackrel{\text{\textsc{re}}}{\rightarrow}$};
\node() at (22,-1.5){$\downarrow\text{\scriptsize \textsc{re}}$};
\node() at (18,-7.5){$\stackrel{\text{\textsc{ybe}}}{\leftarrow}$};
\end{tikzpicture}\caption{Proof of Lemma \ref{lem:extend_reflection}. The picture shows the case in which the two sides of the \textsc{re} are applied to a pair whose first entry has length $2$ and whose second entry has length $1$.}\label{fig:proof}
\end{center}
\end{figure}
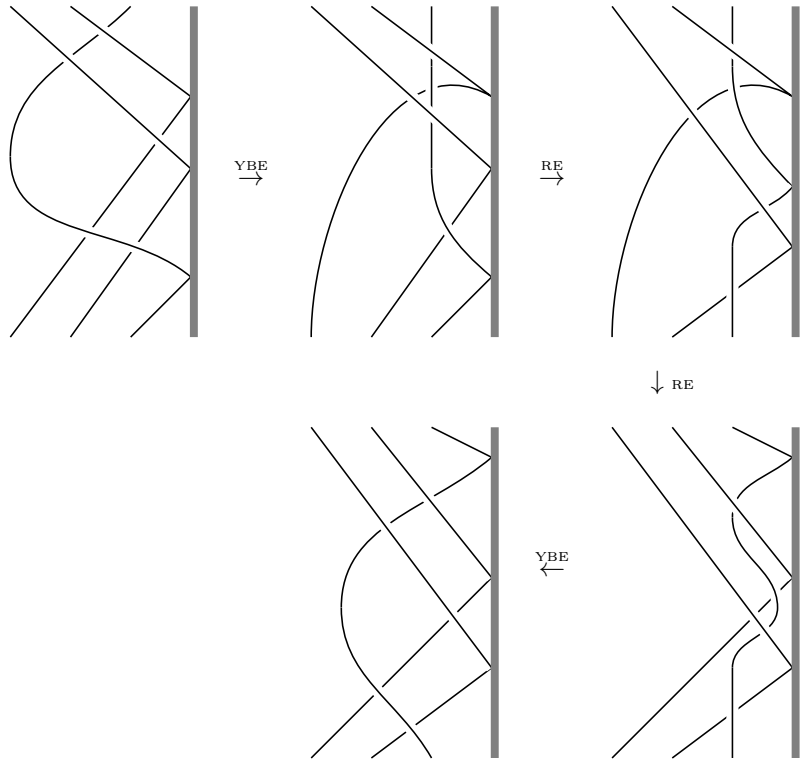

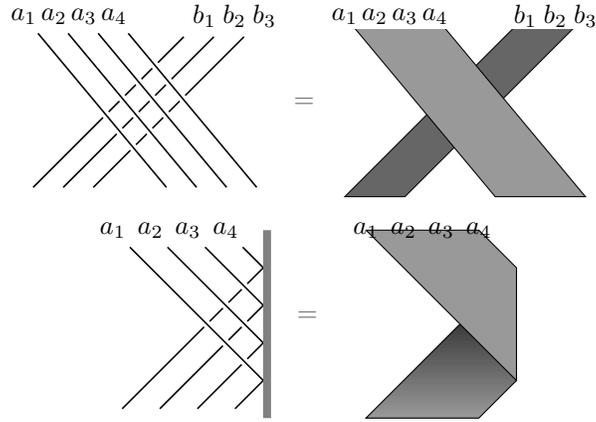
\begin{figure}[t]
\begin{center}
\begin{tikzpicture}[x=0.4cm, y=0.4cm]
\node (a1) at (-4,3) {$a_1$};
\node (a2) at (-3,3) {$a_2$};
\node (a3) at (-2,3) {$a_3$};
\node (a4) at (-1,3) {$a_4$};
\node (b1) at (2,3) {$b_1$};
\node (b2) at (3,3) {$b_2$};
\node (b3) at (4,3) {$b_3$};
\node (B1) at (-4,-3) {};
\node (B2) at (-3,-3) {};
\node (B3) at (-2,-3) {};
\node (A1) at (1,-3) {};
\node (A2) at (2,-3) {};
\node (A3) at (3,-3) {};
\node (A4) at (4,-3) {};
\draw[draw=white, double=black, very thick] (b1) to (B1);
\draw[draw=white, double=black, very thick] (b2) to (B2);
\draw[draw=white, double=black, very thick] (b3) to (B3);
\draw[draw=white, double=black, very thick] (a1) to (A1);
\draw[draw=white, double=black, very thick] (a2) to (A2);
\draw[draw=white, double=black, very thick] (a3) to (A3);
\draw[draw=white, double=black, very thick] (a4) to (A4);
\end{tikzpicture}
\raisebox{1.3cm}{=}
\begin{tikzpicture}[x=0.4cm, y=0.4cm]
\filldraw[fill=gray3, draw opacity=0] (2,3) -- (4,3) -- (-2,-3)-- (-4,-3)-- (2,3);
\filldraw[fill=gray2, draw opacity=0] (-4,3) -- (-1,3) -- (4,-3)-- (1,-3)-- (-4,3);
\filldraw[fill=white, draw=white] (-4.2, 3.2)--(4.2,3.2)--(4.2,2.6)--(-4.2,2.6)--(-4.2,3.2);
\node (a1) at (-4,3) {$a_1$};
\node (a2) at (-3,3) {$a_2$};
\node (a3) at (-2,3) {$a_3$};
\node (a4) at (-1,3) {$a_4$};
\node (b1) at (2,3) {$b_1$};
\node (b2) at (3,3) {$b_2$};
\node (b3) at (4,3) {$b_3$};
\node (B1) at (-4,-3) {};
\node (B2) at (-3,-3) {};
\node (B3) at (-2,-3) {};
\node (A1) at (1,-3) {};
\node (A2) at (2,-3) {};
\node (A3) at (3,-3) {};
\node (A4) at (4,-3) {};
\end{tikzpicture}
\end{center}

\begin{center}
\begin{tikzpicture}[x=0.5cm, y=0.5cm]
\node (a1) at (-2,3) {$a_1$};
\node (a2) at (-1,3) {$a_2$};
\node (a3) at (0,3) {$a_3$};
\node (a4) at (1,3) {$a_4$};
\node (A1) at (1,-2) {};
\node (A2) at (0,-2) {};
\node (A3) at (-1,-2) {};
\node (A4) at (-2,-2) {};
\draw[draw=white, double=black, very thick] (a4)--(2,2)--(A4);
\draw[draw=white, double=black, very thick] (a3)--(2,1)--(A3);
\draw[draw=white, double=black, very thick] (a2)--(2,0)--(A2);
\draw[draw=white, double=black, very thick] (a1)--(2,-1)--(A1);
\draw[draw=gray, line width=3pt] (2.1,3) to (2.1,-2);
\end{tikzpicture}\quad \raisebox{1.4cm}{=}\quad 
\begin{tikzpicture}[x=0.5cm, y=0.5cm]
\filldraw[draw opacity=0, bottom color=gray2, top color=black] (2,2) -- (2,-1) -- (1,-2)-- (-2,-2)-- (2,2);
\filldraw[fill=gray2, draw opacity=0] (-2,3) -- (1,3) -- (2,2)-- (2,-1)-- (-2,3);
\filldraw[fill=white, draw=white] (-2.2, 3.2)--(1.6,3.2)--(1.6,2.6)--(-2.2,2.6)--(-2.2,3.2);
\node (a1) at (-2,3) {$a_1$};
\node (a2) at (-1,3) {$a_2$};
\node (a3) at (0,3) {$a_3$};
\node (a4) at (1,3) {$a_4$};
\node (A1) at (1,-2) {};
\node (A2) at (0,-2) {};
\node (A3) at (-1,-2) {};
\node (A4) at (-2,-2) {};
\end{tikzpicture}
\end{center}
\caption{Ribbon notation for the solution and the reflection in $M(X,r)$. From now on, in depicting reflections, we will drop the `wall' on the right, because the ribbon notation removes the ambiguity.}\label{fig:ribbon}\end{figure}

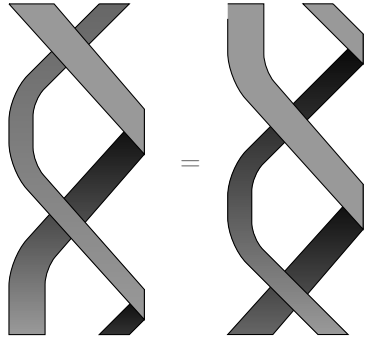
\begin{figure}[h]
\begin{center}
\begin{tikzpicture}[x=0.2cm, y=0.2cm]
\filldraw[draw opacity=0, top color=black, bottom color=gray4] (5,-16)--(5,-18)--(4,-19)--(2,-19)--(5,-16);
\filldraw[fill=gray2, draw opacity=0, top color=black, bottom color=gray2] (5,-4) -- (5,-7) [rounded corners=6pt]-- (-1.6,-14.6)[sharp corners]-- (-1.6,-19)--(-4,-19)[rounded corners=10pt]--(-4,-14)-- (5,-4);
\filldraw[draw opacity=0, top color=gray3, bottom color=gray2] (2,3)--(4,3)[rounded corners=6pt]--(-2.4,-3.4)--(-2.4,-7.6)[sharp corners]--(5, -16)--(5,-18)[rounded corners=10pt]--(-4,-8)--(-4,-3)--(2,3);
\filldraw[fill=gray2, draw opacity=0] (-4,3)--(-1,3)--(5,-4)--(5,-7)--(-4,3);
\end{tikzpicture}
\quad \raisebox{2.2cm}{=}\quad 
\begin{tikzpicture}[x=0.2cm, y=0.2cm]
\filldraw[fill=gray2, draw opacity=0, top color=black, bottom color=gray3] (5,-8)--(5,-11)--(-1,-18)--(-4,-18)--(5,-8);
\filldraw[draw opacity=0, bottom color=gray2, top color=black] (5,2)--(5,0)[rounded corners=6pt]--(-2.4,-7.4)--(-2.4,-11.6)[sharp corners]--(4,-18)--(2,-18)[rounded corners=10pt]--(-4,-12)--(-4,-7)--(5,2);
\filldraw[fill=gray2, draw opacity=0] (3,4)--(5,2)--(5,0)--(1,4)--(3,4);
\filldraw[fill=gray2, draw opacity=0] (-4,4)--(-1.6,4)[rounded corners=6pt]--(-1.6,-0.4)[sharp corners]--(5,-8)--(5,-11)[rounded corners=10pt]--(-4,-1)--(-4,3);
\end{tikzpicture}
\end{center}
\caption{Geometric interpretation of the reflection equation in $M(X,r)$, as a rule that makes the folds slide under each other.} \label{fig:reflection}\end{figure}

\begin{remark}\label{rem:kbar_respects_product}Observe that $\bar{k}$ satisfies \eqref{eq:reflection-unity} by construction, and \eqref{eq:reflection-product} by \cite[Theorem 1.8]{lebed2022reflection}. Indeed, \eqref{eq:reflection-product} corresponds to the fact that $\tilde{r}$ passes to a well defined map in the quotient. A graphical interpretation is reported in Figure \ref{fig:braided-refl}.

In general, $\bar{k}$ need not satisfy \eqref{eq:weird}: if $\bar{k}$ does, then in particular \eqref{eq:weird} holds in degree 1, so $k$ satisfies \eqref{eq:weird} as well, and this is generally not true (recall for instance that \eqref{eq:weird} holds for $k=\id$ if and only if $r$ is involutive). \end{remark}
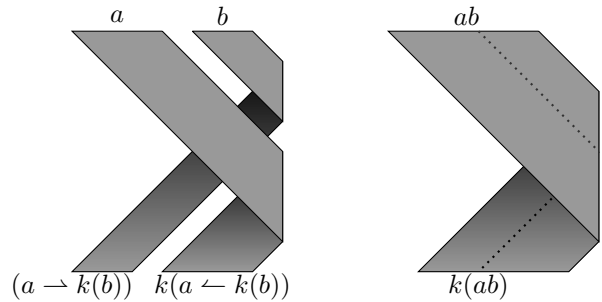
\begin{figure}[t]
\begin{center}
\begin{tikzpicture}[x=0.4cm, y=0.4cm]
\filldraw[draw opacity=0, bottom color=gray2, top color=black] (7,3)--(7,1)--(2,-4)--(0,-4)--(7,3);
\filldraw[draw opacity=0, bottom color=gray2, top color=black] (7,0)--(7,-3)--(6,-4)--(3,-4)--(7,0);
\filldraw[draw opacity=0, fill=gray2] (4,4)--(6,4)--(7,3)--(7,1)--(4,4);
\filldraw[draw opacity=0, fill=gray2] (3,4)--(7,0)--(7,-3)--(0,4)--(3,4);
\node() at (0,-4.5) {$ (a\rightharpoonup k(b))$};
\node() at (5,-4.5) {$ k(a\leftharpoonup k(b))$};
\node() at (1.5, 4.5) {$a$};
\node() at (5, 4.5) {$b$};
\end{tikzpicture}
\qquad\quad
\begin{tikzpicture}[x=0.4cm, y=0.4cm]
\filldraw[draw opacity=0, bottom color=gray2, top color=black] (7,2)--(7,-3)--(6,-4)--(1,-4)--(7,2);
\draw[thick, dotted] (7,0)--(3,-4);
\filldraw[draw opacity=0, fill=gray2] (5,4)--(7,2)--(7,-3)--(0,4)--(5,4);
\draw[thick, dotted,draw=gray4] (3,4)--(7,0);
\node() at (3,-4.5) {$ k(ab)$};
\node() at (2.6, 4.5) {$ab$};
\end{tikzpicture}
\end{center}
\caption{Graphic interpretation of \eqref{eq:reflection-product} in a braided group, as the consistency of a ribbon notation. The same picture justifies why this property holds true for $\bar{k}$ in $M(X,r)$.}\label{fig:braided-refl}
\end{figure}
\begin{lemma}\label{lem:weird}
The map $\bar{k}$ satisfies \eqref{eq:weird} if and only if $k$ satisfies \eqref{eq:weird}.
\end{lemma}
\begin{proof}
If $\bar{k}$ satisfies \eqref{eq:weird}, then in particular this holds in degree $1$. Conversely, if \eqref{eq:weird} holds in degree $1$, observe that $M(X,r)$ is generated by the elements of degree $1$: thus the conclusion follows by Remark \ref{rem:k|_X}, where the product in the monoid is induced by the juxtaposition of words.
\end{proof}

\begin{theorem}\label{thm:extensiontogroup_k}
	Let $k$ satisfy \eqref{eq:weird}. Then the map $\tilde{k}$ on $X^*$ can be extended to a map $\tilde{k}$ on $\Free(X)$, which in turn induces a group reflection $\bar{k}$ on $G(X,r)$.
\end{theorem}
\begin{proof}
	We denote by $a$ an element of $X$, and by $a^{-1}$ an element of the set of formal inverses $X^{-1}$. We first try to extend $\tilde{k}$ to the set $(X\sqcup X^{-1})^*$: this extension will be ambiguously defined, but the ambiguity will disappear in the quotient $G(X,r)$.
	
	\subsubsection*{Definition of $\bar{k}$} A word is \textit{reduced} if it does not contain occurrences of $a a^{-1}$ or $a^{-1} a$ for any $a\in X$. If $a_1^{\varepsilon_1}a_2^{\varepsilon_2}$ is a reduced word, $\varepsilon_i = \pm 1$, we define $\tilde{k}(a_i^{\varepsilon_i}) = k(a_i)^{\varepsilon_i}$, and 
	\[  \tilde{k}(a_1^{\varepsilon_1}a_2^{\varepsilon_2}) = \left(a_1^{\varepsilon_1} \,\tilde{\rightharpoonup}\, k(a_2)^{\varepsilon_2}\right) \tilde{k}\left(  a_1^{\varepsilon_1} \,\tilde{\leftharpoonup}\, k(a_2)^{\varepsilon_2}\right).\]
	We would like to extend this definition to words of any length, by imposing \eqref{eq:reflection-product}. However, as we shall see, this is generally \textit{not} consistent with the associativity on words of length three. But it becomes consistent modulo the relations $xy \sim (x\,\tilde{\rightharpoonup}\, y)(x\,\tilde{\leftharpoonup}\, y)$, i.e., the map $\bar{k}$ on $G(X,r)$ will be well-defined and will satisfy \eqref{eq:reflection-product}.
	
	\subsubsection*{Twisted involutivity for words of length one} Assume that $k$ satisfies $\eqref{eq:weird}$. First of all, we prove that $\tilde{k}$ (and hence $\bar{k}$) satisfies \eqref{eq:weird} on the words of length 1.
	
	We first observe that the extension of $r$ is defined so that
	\[  (a_1^{\varepsilon_1}\rightharpoondown a_2^{\varepsilon_2} )^{-1}  = a_2^{-\varepsilon_2}\leftharpoonup a_1^{-\varepsilon_1},\quad (a_1^{\varepsilon_1}\leftharpoondown a_2^{\varepsilon_2})^{-1} = a_2^{-\varepsilon_2} \rightharpoonup a_1^{-\varepsilon_1}\]
	holds; and $\tilde{k}$ is defined so that \eqref{eq:inversek} holds on words of length at most two. Thus, with the same proof as in Lemma \ref{lem:inverseproperties}, one has that $\tilde{k}$ satisfies \eqref{eq:right_reflection-product} on words of length at most two, and \eqref{eq:right_weird} on words of length at most one.
	
	 Let $a,b\in X$.\begin{enumerate}
			\item The relation $k(a) = (a\rightharpoonup b)\rightharpoonup k(b\leftharpoonup b)$ holds because $k$ satisfies \eqref{eq:weird}.
			\item The relation $k(a) = (a\,\tilde{\rightharpoonup}\, b^{-1})\,\tilde{\rightharpoonup}\, k(a\,\tilde{\leftharpoonup}\, b^{-1})$ is equivalent to 
			\[  ((a\,\tilde{\leftharpoonup}\, b^{-1})\rightharpoonup b) \rightharpoonup k(a) = k(a\,\tilde{\leftharpoonup}\, b^{-1}), \]
			which is verified using
			\[k(\alpha) =  (\alpha\rightharpoonup \beta)\rightharpoonup k(\alpha\leftharpoonup \beta)
			\]
			with $\alpha = a\,\tilde{\leftharpoonup}\,b^{-1} \in X$ and $\beta = b\in X$.
			\item The relation $\tilde{k}(a^{-1}) = (a^{-1}\,\tilde{\rightharpoonup}\, b^{-1})\,\tilde{\rightharpoonup}\, \tilde{k}(a^{-1}\,\tilde{\leftharpoonup}\, b^{-1})$ follows from
			\begin{align*}
				& (a^{-1}\,\tilde{\rightharpoonup}\, b^{-1})\,\tilde{\rightharpoonup}\, \tilde{k}(a^{-1}\,\tilde{\leftharpoonup}\, b^{-1})\\
				&= (b\leftharpoondown a)^{-1} \rightharpoonup k(b\rightharpoondown a)^{-1}\\
				&= \left(  k(b\rightharpoondown a)\leftharpoondown (b\leftharpoondown a) \right)^{-1}\\
				\overset{\eqref{eq:weird}}&{=} k(a)^{-1}
			\end{align*}
			using \eqref{eq:right_weird} on the words of length one.
			\item The relation $\tilde{k}(a^{-1}) = (a^{-1}\,\tilde{\rightharpoonup}\, b)\rightharpoonup k(a^{-1}\,\tilde{\leftharpoonup}\, b)$ is seen to be equivalent to 
			\[  k(a) \leftharpoondown (b\leftharpoondown (b^{-1}\rightharpoondown a))   = k(b^{-1}\rightharpoondown a),\]
			which is verified using
			\[ k(\beta) = k(\alpha\rightharpoondown \beta) \leftharpoondown (\alpha\leftharpoondown \beta)  \]
			with $\alpha = b$ and $\beta = b^{-1}\rightharpoondown a$.
		\end{enumerate} 
	
	\subsubsection*{Extension of $\bar{k}$ to words of any length}Since $(\tilde{\rightharpoonup}\,,\,\tilde{\leftharpoonup})$ defines a braiding on $\Free(X)$ \cite{LYZ}, assuming that all words at each step are reduced, one has
	\begin{align*}
		\tilde{k}(a_1^{\varepsilon_1}(a_2^{\varepsilon_2}a_3^{\varepsilon_3})) &= \left( a_1^{\varepsilon_1} \,\tilde{\rightharpoonup}\, \tilde{k}(a_2^{\varepsilon_2}a_3^{\varepsilon_3})  \right) \tilde{k} \left( a_1^{\varepsilon_1}\,\tilde{\leftharpoonup}\,\tilde{k}(a_2^{\varepsilon_2}a_3^{\varepsilon_3})  \right) \\
		\overset{(\dagger)}&{=} \left( a_1^{\varepsilon_1}\,\tilde{\rightharpoonup}\, \left(\left(  a_2^{\varepsilon_2} \,\tilde{\rightharpoonup}\, k(a_3)^{\varepsilon_3} \right) \tilde{k} \left(  a_2^{\varepsilon_2} \,\tilde{\leftharpoonup}\, k(a_3)^{\varepsilon_3} \right) \right)  \right)  \\ 
		&\hspace{2em}\tilde{k} \left( a_1^{\varepsilon_1}\,\tilde{\leftharpoonup}\, \left(\left(  a_2^{\varepsilon_2} \,\tilde{\rightharpoonup}\, k(a_3)^{\varepsilon_3} \right) \tilde{k} \left(  a_2^{\varepsilon_2} \,\tilde{\leftharpoonup}\, k(a_3)^{\varepsilon_3} \right) \right)  \right) \\
		&= \left( a_1^{\varepsilon_1}a_2^{\varepsilon_2}\,\tilde{\rightharpoonup}\, k(a_3)^{\varepsilon_3}\right)  \left( \left(a_1^{\varepsilon_1}\,\tilde{\leftharpoonup}\,    \left(a_2^{\varepsilon_2} \,\tilde{\rightharpoonup}\, k(a_3)^{\varepsilon_3}\right)  \right) \,\tilde{\rightharpoonup}\, \tilde{k}\left(  a_2^{\varepsilon_2} \,\tilde{\leftharpoonup}\, k(a_3)^{\varepsilon_3}  \right)    \right)\\
		& \hspace{2em} \tilde{k} \left( \left(a_1^{\varepsilon_1}\,\tilde{\leftharpoonup}\,    \left(a_2^{\varepsilon_2} \,\tilde{\rightharpoonup}\, k(a_3)^{\varepsilon_3}\right)  \right) \,\tilde{\leftharpoonup}\, \tilde{k}\left(  a_2^{\varepsilon_2} \,\tilde{\leftharpoonup}\, k(a_3)^{\varepsilon_3}  \right)    \right)\\
		\overset{(\ddagger)}&{=} \left( a_1^{\varepsilon_1}a_2^{\varepsilon_2}\,\tilde{\rightharpoonup}\, k(a_3)^{\varepsilon_3}\right) \tilde{k}\left( a_1^{\varepsilon_1}a_2^{\varepsilon_2}\,\tilde{\leftharpoonup}\, k(a_3)^{\varepsilon_3}  \right)\\
		&= \tilde{k} ((a_1^{\varepsilon_1}a_2^{\varepsilon_2})a_3^{\varepsilon_3}).
	\end{align*}
	The only steps where we use that all words are reduced, are the ones marked with $(\dagger), (\ddagger)$. We check what happens if some of the words are not reduced.
	
	\begin{enumerate}\item If the word $a_2^{\varepsilon_2}a_3^{\varepsilon_3}$ is not reduced, which means $a_3^{\varepsilon_3}=(a_2^{\varepsilon_2})^{-1}$, then the step $(\dagger)$ is not valid any more. We assume that the word
	\[  \left(a_1^{\varepsilon_1}\,\tilde{\leftharpoonup}\,    \left(a_2^{\varepsilon_2} \,\tilde{\rightharpoonup}\, k(a_3)^{\varepsilon_3}\right)  \right) \left(  a_2^{\varepsilon_2} \,\tilde{\leftharpoonup}\, k(a_3)^{\varepsilon_3}  \right)  = (a_1^{\varepsilon_1}a_2^{\varepsilon_2})\,\tilde{\leftharpoonup}\, (k(a_3)^{\varepsilon_3})^{-1}\]
	is reduced, i.e.\@ that $a_1^{\varepsilon_1}a_2^{\varepsilon_2}$ is reduced, so that the step $(\ddagger)$ is valid. On the one hand, we have $\tilde{k} (a_1^{\varepsilon_1} (a_2^{\varepsilon_2} a_3^{\varepsilon_3}) ) = k(a_1)^{\varepsilon_1}$, and on the other hand
	\begin{align*}
		&\tilde{k}((a_1^{\varepsilon_1} a_2^{\varepsilon_2})a_3^{\varepsilon_3})\\
		&= \left( a_1^{\varepsilon_1}a_2^{\varepsilon_2}\,\tilde{\rightharpoonup}\, k(a_3)^{\varepsilon_3}\right) \tilde{k}\left( a_1^{\varepsilon_1}a_2^{\varepsilon_2}\,\tilde{\leftharpoonup}\, k(a_3)^{\varepsilon_3}  \right)\\
		&= \left( a_1^{\varepsilon_1}a_2^{\varepsilon_2}\,\tilde{\rightharpoonup}\, k(a_3)^{\varepsilon_3}\right)   \tilde{k} \left( \left(a_1^{\varepsilon_1}\,\tilde{\leftharpoonup}\,    \left(a_2^{\varepsilon_2} \,\tilde{\rightharpoonup}\, k(a_3)^{\varepsilon_3}\right)  \right) \left(  a_2^{\varepsilon_2} \,\tilde{\leftharpoonup}\, k(a_3)^{\varepsilon_3}  \right)\right)\\
		\overset{(\ddagger)}&{=} \left( a_1^{\varepsilon_1}a_2^{\varepsilon_2}\,\tilde{\rightharpoonup}\, k(a_3)^{\varepsilon_3}\right)  \left( \left(a_1^{\varepsilon_1}\,\tilde{\leftharpoonup}\,    \left(a_2^{\varepsilon_2} \,\tilde{\rightharpoonup}\, k(a_3)^{\varepsilon_3}\right)  \right) \,\tilde{\rightharpoonup}\, \tilde{k}\left(  a_2^{\varepsilon_2} \,\tilde{\leftharpoonup}\, k(a_3)^{\varepsilon_3}  \right)    \right)\\
		& \hspace{2em} \tilde{k} \left( \left(a_1^{\varepsilon_1}\,\tilde{\leftharpoonup}\,    \left(a_2^{\varepsilon_2} \,\tilde{\rightharpoonup}\, k(a_3)^{\varepsilon_3}\right)  \right) \,\tilde{\leftharpoonup}\, \tilde{k}\left(  a_2^{\varepsilon_2} \,\tilde{\leftharpoonup}\, k(a_3)^{\varepsilon_3}  \right)    \right)\\
		&= \left( a_1^{\varepsilon_1}\,\tilde{\rightharpoonup}\, \left(\left(  a_2^{\varepsilon_2} \,\tilde{\rightharpoonup}\, k(a_3)^{\varepsilon_3} \right) \tilde{k} \left(  a_2^{\varepsilon_2} \,\tilde{\leftharpoonup}\, k(a_3)^{\varepsilon_3} \right) \right)  \right)  \\ 
		&\hspace{2em}\tilde{k} \left( a_1^{\varepsilon_1}\,\tilde{\leftharpoonup}\, \left(\left(  a_2^{\varepsilon_2} \,\tilde{\rightharpoonup}\, k(a_3)^{\varepsilon_3} \right) \tilde{k} \left(  a_2^{\varepsilon_2} \,\tilde{\leftharpoonup}\, k(a_3)^{\varepsilon_3} \right) \right)  \right) \\
		&= (a_1^{\varepsilon_1}\,\tilde{\rightharpoonup}\, 1) \tilde{k}(a_1^{\varepsilon_1}\,\tilde{\leftharpoonup}\,1)\\
		&= k(a_1)^{\varepsilon_1},
	\end{align*}
	as desired. 
	\item If the word 
	\[  \left(a_1^{\varepsilon_1}\,\tilde{\leftharpoonup}\,    \left(a_2^{\varepsilon_2} \,\tilde{\rightharpoonup}\, k(a_3)^{\varepsilon_3}\right)  \right) \left(  a_2^{\varepsilon_2} \,\tilde{\leftharpoonup}\, k(a_3)^{\varepsilon_3}  \right)  = (a_1^{\varepsilon_1}a_2^{\varepsilon_2})\,\tilde{\leftharpoonup}\, (k(a_3)^{\varepsilon_3})^{-1}\]
	is not reduced, i.e.\@ if $(a_1^{\varepsilon_1})^{-1}=a_2^{\varepsilon_2}$, then the step $(\ddagger)$ is not valid anymore. We we assume that the word $a_2^{\varepsilon_2}a_3^{\varepsilon_3}$ is reduced, i.e., that the step $(\dagger)$ is valid. On the one hand,
	\[ \tilde{k}((a_1^{\varepsilon_1}a_2^{\varepsilon_2})a_3^{\varepsilon_3}) = k(a_3)^{\varepsilon_3}.\]
	On the other hand, 
		\begin{align*}
		&\tilde{k}(a_1^{\varepsilon_1}(a_2^{\varepsilon_2}a_3^{\varepsilon_3})) \\ &= \left( a_1^{\varepsilon_1} \,\tilde{\rightharpoonup}\, \tilde{k}(a_2^{\varepsilon_2}a_3^{\varepsilon_3})  \right) \tilde{k} \left( a_1^{\varepsilon_1}\,\tilde{\leftharpoonup}\,\tilde{k}(a_2^{\varepsilon_2}a_3^{\varepsilon_3})  \right) \\
		\overset{(\dagger)}&{=} \left( a_1^{\varepsilon_1}\,\tilde{\rightharpoonup}\, \left(\left(  a_2^{\varepsilon_2} \,\tilde{\rightharpoonup}\, k(a_3)^{\varepsilon_3} \right) \tilde{k} \left(  a_2^{\varepsilon_2} \,\tilde{\leftharpoonup}\, k(a_3)^{\varepsilon_3} \right) \right)  \right)  \\ 
		&\hspace{2em}\tilde{k} \left( a_1^{\varepsilon_1}\,\tilde{\leftharpoonup}\, \left(\left(  a_2^{\varepsilon_2} \,\tilde{\rightharpoonup}\, k(a_3)^{\varepsilon_3} \right) \tilde{k} \left(  a_2^{\varepsilon_2} \,\tilde{\leftharpoonup}\, k(a_3)^{\varepsilon_3} \right) \right)  \right) \\
		&= (a_1^{\varepsilon_1}a_2^{\varepsilon_2}\,\tilde{\rightharpoonup}\, k(a_3)^{\varepsilon_3})  \left( (a_1^{\varepsilon_1} \,\tilde{\leftharpoonup}\, (a_2^{\varepsilon_2}\,\tilde{\rightharpoonup}\, k(a_3)^{\varepsilon_3})) \,\tilde{\rightharpoonup}\, \tilde{k}(a_2^{\varepsilon_2}\,\tilde{\leftharpoonup}\, k(a_3)^{\varepsilon_3}) \right)\\
		& \hspace{2em}  \tilde{k} \left( a_1^{\varepsilon_1}\,\tilde{\leftharpoonup}\, \left(\left(  a_2^{\varepsilon_2} \,\tilde{\rightharpoonup}\, k(a_3)^{\varepsilon_3} \right) \tilde{k} \left(  a_2^{\varepsilon_2} \,\tilde{\leftharpoonup}\, k(a_3)^{\varepsilon_3} \right) \right)  \right).
	\end{align*}
	Now observe that $\tilde{k}(a^\varepsilon a^{-\varepsilon})=1$ implies \[\tilde{k}(a^\varepsilon\,\tilde{\leftharpoonup}\, k(a)^{-\varepsilon}) ((a^\varepsilon\,\tilde{\leftharpoonup}\, k(a)^{-\varepsilon})\,\tilde{\rightharpoonup}\, k(a)^\varepsilon) = 1\] by \eqref{eq:weird} on words of length one.
	\item Finally, we assume that both the words $a_1^{\varepsilon_1}a_2^{\varepsilon_2}$ and $a_2^{\varepsilon_2}a_3^{\varepsilon_3}$
	are not reduced, hence neither $(\dagger)$ nor $(\ddagger)$ is valid. This implies $a_1^{\varepsilon_1} = (a_2^{\varepsilon_2})^{-1} = a_3^{\varepsilon_3}$, therefore
	\[ \tilde{k} ((a_1^{\varepsilon_1}a_2^{\varepsilon_2})a_3^{\varepsilon_3}) = k(a_3)^{\varepsilon_3} = k(a_1)^{\varepsilon_1} = \tilde{k}(a_1^{\varepsilon_1}(a_2^{\varepsilon_2} a_3^{\varepsilon_3})), \]
	as desired.
	\end{enumerate}
	We have proven the relation \[\tilde{k}(a_1^{\varepsilon_1}(a_2^{\varepsilon_2} a_3^{\varepsilon_3})) \sim \tilde{k}((a_1^{\varepsilon_1}a_2^{\varepsilon_2}) a_3^{\varepsilon_3}).\]
	Therefore, if $\tilde{k}$ passes to a well-defined map $\bar{k}$ on the classes of words of length at most two in $G(X,r)$, then this map $\bar{k}$ can be extended to the whole of $G(X,r)$ consistently, by imposing \eqref{eq:reflection-product}. However, we still need to prove that $\tilde{k}$ passes to such a map $\bar{k}$.
		\subsubsection*{Inducing the map to the quotient} We then check that $\tilde{k}$ is compatible with the relations \[a_1^{\varepsilon_1} a_2^{\varepsilon_2} \sim (a_1^{\varepsilon_1}\,\tilde{\rightharpoonup}\, a_2^{\varepsilon_2})(a_1^{\varepsilon_1}\,\tilde{\leftharpoonup}\, a_2^{\varepsilon_2}).\]
		Knowing that $\tilde{k}$ is a set-theoretic reflection on $\Free(X)$, one has:
	\begin{align*}
		&\tilde{k} (a_1^{\varepsilon_1}a_2^{\varepsilon_2})\\
		&= (a_1^{\varepsilon_1}\,\tilde{\rightharpoonup}\, k(a_2)^{\varepsilon_2}) \tilde{k}(a_1^{\varepsilon_1}\,\tilde{\leftharpoonup}\, k(a_2)^{\varepsilon_2})\\
		&\sim \left( (a_1^{\varepsilon_1}\,\tilde{\rightharpoonup}\, k(a_2)^{\varepsilon_2})\,\tilde{\rightharpoonup}\,\tilde{k}(a_1^{\varepsilon_1}\,\tilde{\leftharpoonup}\, k(a_2)^{\varepsilon_2})    \right) \left( (a_1^{\varepsilon_1}\,\tilde{\rightharpoonup}\, k(a_2)^{\varepsilon_2})\,\tilde{\leftharpoonup}\,\tilde{k}(a_1^{\varepsilon_1}\,\tilde{\leftharpoonup}\, k(a_2)^{\varepsilon_2})   \right)\\
		\overset{\eqref{re}}&{=} \tilde{k}\left( (a_1^{\varepsilon_1} \,\tilde{\rightharpoonup}\, a_2^{\varepsilon_2}) (a_1^{\varepsilon_1} \,\tilde{\leftharpoonup}\, a_2^{\varepsilon_2})\right),
	\end{align*}
	as desired. 
		
	\subsubsection*{Conclusion} By definition of $\tilde{k}$, it is clear that the induced map $\bar{k}$ on $G(X,r)$ satisfies \eqref{eq:reflection-unity} and \eqref{eq:reflection-product} on words of length at most two. Since $\bar{k}$ has been extended to all words through \eqref{eq:reflection-product}, and we have observed that this definition is consistent, it is clear that $\bar{k}$ satisfies \eqref{eq:reflection-product} on the whole of $G(X,r)$.

We have proven \eqref{eq:weird} for $\bar{k}$ on a set of generators for $G(X,r)$ as a monoid; namely, the set $X\sqcup X^{-1}$. As observed in Lemma \ref{lem:weird}, then, \eqref{eq:weird} can be proven on words of any length, by double induction. Therefore, $\bar{k}$ is a group reflection. By its definition, it also satisfies \eqref{eq:inversek}.
\end{proof}
\begin{remark}
	If a group is abelian, then its group reflections are anti-homomorphisms (Remark \ref{rem:trivialbrace}). It is known that every map $f\colon X\to X$ can be extended to an anti-homomorphism on $\Free(X)$, and this extension is a subcase of our definition of $\tilde{k}$ in Theorem \ref{thm:extensiontogroup_k}, where both actions $\tilde{\rightharpoonup}, \,\tilde{\leftharpoonup}$ are trivial.
\end{remark}

The problem of whether a Drinfeld twist for $(X,r)$ can be extended to a group Drinfeld twist on $G(X,r)$ has already been posed by Ghobadi in \cite[Example 2.4]{ghobadi2021drinfeld}. Theorem \ref{thm:extensiontogroup_k} says that we can extend $k$ to $\bar{k}$ if $k$ satisfies \eqref{eq:weird}. On the other hand, if $\bar{k}$ is an extension of $k$ to $G(X,r)$, and it satisfies \eqref{eq:weird}, then in degree one $k$ must satisfy \eqref{eq:weird}: thus \eqref{eq:weird} is a necessary and sufficient condition for $k$ to be extended to a group reflection on $G(X,r)$. This solves Ghobadi's problem \cite[Example 2.4]{ghobadi2021drinfeld} in the case of Drinfeld twists coming from reflections.

\medskip

\noindent\textbf{Acknowledgements.} This work originated from a conversation with Victoria Lebed and Leandro Vendramin in Brussels, in November 2023. The author is especially  grateful to Victoria Lebed for suggesting that the $k$-Garside map could be a reflection on the structure group $G(X,r)$, and for many other insightful remarks; to an anonymous referee for some suggestions and improvements; to Leandro Vendramin for his constant support and advice; to Alessandro Ardizzoni for his feedback on this paper; and to Jonathan D.\@ H.\@ Smith for his interest and remarks. The content of \S\ref{sec:representations} was inspired by a question of Victoria Lebed, and was significantly improved after a conversation with Silvia Properzi and Charlotte Roelants. Proposition \ref{prop:iso_representations} was strengthened by a remark of Silvia Properzi, to whom the author is grateful.

During part of this research, the author was a member of the \textit{Istituto Nazionale di Alta Matematica} (INdAM) and supported by the PRIN 2022 of the Italian Ministry of University and Research, CUP D53D23005960006 Project 2022S97PMY. 

This research was supported by the University of Turin through a PNRR DM 118 scholarship; by the project
OZR3762 of the Vrije Universiteit Brussel; and through the
FWO Senior Research Project G004124N.
\medskip

\noindent\textbf{Statements and declarations.} The author has no potential conflict of interest to disclose. 

No new data has been produced in this research. The database \texttt{SmallSkewbrace} used in Example \ref{ex:ell_not_refl} is distributed with \textsc{gap} 4.13.0 \cite{gap4}, and available at \url{https://github.com/vendramin/YangBaxter}.
\medskip

\bibliographystyle{acm}
\bibliography{refs}

\end{document}